\documentclass[11pt,a4paper]{amsart}
\newcommand{\Title}{Kadison--Kastler Distance and Interior Angle for Masas: Connections with Entropy and Probabilistic Index}%
\newcommand{\ShortTitle}{\Title}%

\newcommand{\AuthorOne}{Indrajit Ghosh}%
\newcommand{\AuthorOneAddr}{%
	Department of Mathematics, Indian Institute of Technology Kanpur, Uttar Pradesh 208 016, India
}%

\newcommand{\AuthorOneEmail}{%
	indrajitghosh912@gmail.com, indrajitg@iitk.ac.in
}%

\newcommand{\AuthorTwo}{Sumit Kumar}%
\newcommand{\AuthorTwoAddr}{Department of Mathematics, Indian Institute of Technology Kanpur, Uttar Pradesh 208016, India
}%
\newcommand{\AuthorTwoEmail}{sumitkumar.sk809@gmail.com}

\newcommand{\SubjectClassText}{Primary 46L05, 47C15, 47L40; Secondary 46L10}
\newcommand{\Dedicatory}{This paper is dedicated to our PhD advisors.}
\newcommand{\Keywords}{Kadison-Kastler Distance, Interior Angle, Connes-Stormer Entropy, Popa's Probabilistic Index, Hadamard Unitary, Flat Unitary}

\newcommand{\pdfTitle}{\Title}
\newcommand{\pdfAuthor}{Indrajit Ghosh}
\newcommand{\pdfSubject}{Mathematics, Research paper}
\newcommand{\pdfKeywords}{\Keywords}
\newcommand{\pdfCreator}{TeXlive}
\newcommand{\pdfCreationDate}{\today}
\newcommand{\pdfColorLink}{true}
\newcommand{\pdfLinkColor}{cyan}
\newcommand{\pdfUrlColor}{blue}
\newcommand{\pdfCiteColor}{magenta}

\usepackage[top=0.9in, bottom=1in, left=0.7in, right=0.7in]{geometry}
\usepackage{amsmath,amssymb,amsthm} 
\usepackage[utf8]{inputenc}
\usepackage[T1]{fontenc}
\usepackage{mathtools}
\usepackage{mathrsfs} 
\usepackage{xfrac} 
\usepackage{dsfont} 
\usepackage{array}
\usepackage{verbatim}
\usepackage{graphicx}
\usepackage{mdframed}
\usepackage{enumitem} 
\usepackage{hyperref}
\hypersetup{
	pdftitle={\pdfTitle},
	pdfauthor={\pdfAuthor},
	pdfsubject={\pdfSubject},
	pdfcreationdate={\pdfCreationDate},
	pdfcreator={\pdfCreator},
	pdfkeywords={\pdfKeywords},
	colorlinks=\pdfColorLink,
	linkcolor={\pdfLinkColor},
	urlcolor=\pdfUrlColor,
	citecolor=\pdfCiteColor,
	pdfpagemode=UseOutlines,
}
\usepackage{tikz-cd} 
\usepackage{lipsum}
\usetikzlibrary{matrix,arrows}
\usepackage[english]{babel}
\usepackage{lmodern}
\usepackage{bbm} 
\usepackage[dvipsnames]{xcolor}
\usepackage[most]{tcolorbox}
\usepackage{xparse}

\theoremstyle{plain}
\newtheorem{theorem}{Theorem}[section]
\newtheorem{prop}[theorem]{Proposition}
\newtheorem{lem}[theorem]{Lemma}
\newtheorem{cor}[theorem]{Corollary}

\theoremstyle{definition}
\newtheorem{definition}[theorem]{Definition}

\theoremstyle{remark}
\newtheorem{remark}[theorem]{Remark}

\numberwithin{equation}{section}

\newtheoremstyle{ser}
{8pt}
{8pt}
{\it}
{}
{\sf}
{:}
{6mm}
{}

\theoremstyle{ser}

\newtheoremstyle{serr}
{8pt}
{8pt}
{\normalfont}
{}
{\sf}
{.}
{6mm}
{}

\theoremstyle{serr}

\theoremstyle{ser}

\theoremstyle{ser}

\newtheoremstyle{collabquestion}
  {8pt}
  {8pt}
  {\normalfont}
  {}
  {\sffamily\bfseries\color{blue!70!black}}
  {.}
  {.5em}
  {}

\theoremstyle{collabquestion}
\newtheorem{qninner}{Question}

\definecolor{indraRed}{rgb}{0.593, 0.183, 0.183}
\definecolor{indraPink}{rgb}{0.858, 0.188, 0.478}
\definecolor{indraBlue}{rgb}{0, 0.199, 0.398}
\definecolor{madridBlue}{rgb}{0.199, 0.199, 0.695}
\definecolor{metropolisThemeColor}{rgb}{0.105, 0.214, 0.234}
\definecolor{metropolisBarColor}{rgb}{0.984, 0.0.515, 0.015}
\definecolor{UBCblue}{rgb}{0.04706, 0.13725, 0.26667} 
\definecolor{UBCgrey}{rgb}{0.3686, 0.5255, 0.6235} 

\makeatletter
\def\mathcolor#1#{\@mathcolor{#1}}
\def\@mathcolor#1#2#3{%
	\protect\leavevmode
	\begingroup
	\color#1{#2}#3%
	\endgroup
}
\makeatother

\makeatletter
\def\ps@headings{\ps@empty
  \def\@evenhead{\normalfont\scriptsize\hfil \leftmark{}{}\hfil}%
  \def\@oddhead{\normalfont\scriptsize\hfil \rightmark{}{}\hfil}%
  \let\@mkboth\markboth
  \def\@evenfoot{\normalfont\scriptsize\hfil\thepage\hfil}%
  \def\@oddfoot{\normalfont\scriptsize\hfil\thepage\hfil}%
}
\makeatother

\makeatletter
\def\author@andify{%
  \nxandlist {\unskip ,\penalty-1 \space\ignorespaces}%
    {\unskip {} \@@and~}%
    {\unskip \penalty-2 \space \@@and~}%
}
\makeatother

\definecolor{indraRed}{rgb}{0.593, 0.183, 0.183}
\definecolor{indraPink}{rgb}{0.858, 0.188, 0.478}
\definecolor{indraBlue}{rgb}{0, 0.199, 0.398}
\definecolor{madridBlue}{rgb}{0.199, 0.199, 0.695}
\definecolor{metropolisThemeColor}{rgb}{0.105, 0.214, 0.234}
\definecolor{metropolisBarColor}{rgb}{0.984, 0.0.515, 0.015}
\definecolor{UBCblue}{rgb}{0.04706, 0.13725, 0.26667} 
\definecolor{UBCgrey}{rgb}{0.3686, 0.5255, 0.6235} 

\newcommand{\spn}{{\operatorname{span}\,}}

\newcommand{\C}{\mathbb{C}}
\newcommand{\R}{\mathbb{R}}
\newcommand{\N}{\mathbb{N}}

\newcommand{\Bh}{\mathcal{B}(\mathcal{H})} 
\NewDocumentCommand{\mn}{m O{\mathbb{C}}}
{\mathbb{M}_{#1}(#2)} 

     \newcommand{\sA}{\mathcal A}		
     \newcommand{\sB}{\mathcal B}		
     \newcommand{\sC}{\mathcal C}		
     \newcommand{\sD}{\mathcal D}

     \newcommand{\sG}{\mathcal G}		
     \newcommand{\sH}{\mathcal H}		
     		
\newcommand{\fK}{\mathfrak K}     		
     		
     \newcommand{\sL}{\mathcal L}		
     \newcommand{\sM}{\mathcal M}		
     \newcommand{\sN}{\mathcal N}		
     		
     \newcommand{\sP}{\mathcal P}		
     \newcommand{\sQ}{\mathcal Q}		
     		
     \newcommand{\sS}{\mathcal S}		
     		
     \newcommand{\sU}{\mathcal U}

     \newcommand{\sX}{\mathcal X}		
     \newcommand{\sY}{\mathcal Y}		
     \newcommand{\sZ}{\mathcal Z}

\NewDocumentCommand{\nor}{g}
  {\sN\IfValueT{#1}{_{#1}}}

\NewDocumentCommand{\unor}{g}
  {\sU\sN\IfValueT{#1}{_{#1}}}

\NewDocumentCommand{\grnor}{g}
  {\sG\sN\IfValueT{#1}{_{#1}}}

\newcommand{\dkk}{\mathrm{d}_{\mathrm{KK}}}
\newcommand{\Deltan}[1][n]{\Delta^{(#1)}}

\newcommand{\ball}[1]{\left( #1\right)_1}

\begin{document}
	
        \title[\ShortTitle]{\MakeUppercase\Title}
	\author{\AuthorOne}
	\address[Indrajit Ghosh]{\AuthorOneAddr}
	
	\email{\AuthorOneEmail}

    \author{\AuthorTwo}
	\address[Sumit Kumar]{\AuthorTwoAddr}
	\email{\AuthorTwoEmail}

	\date{}
    \dedicatory{\Dedicatory}
	\subjclass{\SubjectClassText}
	
	
	\begin{abstract}
		We study the Kadison--Kastler distance and the interior angle between masas in finite-dimensional matrix algebras. In $\mathbb{M}_2(\mathbb{C})$, we obtain an explicit formula for the Kadison--Kastler distance between two masas and show that every value in $[0,1]$ is attained by $\mathrm{d}_{\textrm{KK}}(\Delta,u\Delta u^*)$. We further prove that the distance attains its maximal value precisely when the relative unitary is a Hadamard unitary. We also establish the identity
\[
\mathrm{d}_{\textrm{KK}}(\mathcal{A},\mathcal{B})=\sin\alpha(\mathcal{A},\mathcal{B})
\]
for any two intermediate subalgebras $\mathbb{C}\subsetneq\mathcal{A},\mathcal{B}\subsetneq\mathbb{M}_2(\mathbb{C})$.
Most notably, for any unitaries $u,v\in\mathbb{M}_2(\mathbb{C})$, we establish the equivalence of maximal Kadison--Kastler distance, maximal interior angle, maximal Connes--St\o rmer modified entropy, and minimal Popa probabilistic index:
\[
\mathrm{d}_{\textrm{KK}}(u\Delta u^*,v\Delta v^*)=1
\iff
h(u\Delta u^*,v\Delta v^*)=\log 2
\iff
\alpha(u\Delta u^*,v\Delta v^*)=\frac{\pi}{2}
\iff
\lambda(u\Delta u^*,v\Delta v^*)=\frac{1}{2}.
\]
For general $\mathbb{M}_n(\mathbb{C})$, we prove a unitary invariance property for the interior angle and derive an explicit formula for the angle between arbitrary masas of the form $u\Delta^{(n)}u^*$ and $v\Delta^{(n)}v^*$. As a consequence, the interior angle is $\frac{\pi}{2}$ precisely when the relative unitary $u^*v$ is a Hadamard unitary. We further show that every value in $\left[0,\frac{\pi}{2}\right]$ is attained as the interior angle between a pair of masas in $\mathbb{M}_n(\mathbb{C})$. 

Finally, we demonstrate several drastic departures in behavior for dimensions $n \geq 3$. First, the interior angle need not distinguish between distinct intermediate subalgebras: it can vanish even when the subalgebras are distinct. In particular, we construct distinct simple intermediate $C^*$-subalgebras with zero interior angle, thereby answering in the negative the question raised in \cite[Remark 3.4]{gupta2024}. Second, while relative Hadamard unitaries still maximize the Kadison--Kastler distance between masas, we show that this maximal distance of $1$ is also achieved by non-Hadamard unitaries, breaking the equivalence observed in $\mathbb{M}_2(\mathbb{C})$. Lastly, we establish that the trigonometric identity $\mathrm{d}_{\textrm{KK}}(\mathcal{A},\mathcal{B})=\sin\alpha(\mathcal{A},\mathcal{B})$ no longer holds in general for $n \geq 3$.

	\end{abstract}

	\maketitle%
        \thispagestyle{empty}%


    \section{Introduction}
    
\noindent Maximal abelian self-adjoint subalgebras (masas) occupy a central position in the study of operator algebras and provide an important setting in which algebraic, geometric, and metric properties can be investigated explicitly. Given two masas in a $C^*$-algebra, it is natural to ask how their relative position can be quantified and to what extent different invariants capture the same underlying structure.
To address this, the Kadison–Kastler distance (\cite{Kadison_Kastler})  provides a natural metric for measuring the proximity of two $C^*$-subalgebras.
More precisely, if $\sA$ and $\sB$ are unital $C^*$-subalgebras of a $C^*$-algebra $\mathcal M$, their Kadison--Kastler distance $\dkk(\sA,\sB)$ measures how closely the elements of one algebra can be approximated by elements of the other.
The study of this distance has played an important role in understanding the stability of operator-algebraic structures under small perturbations (see, for instance, \cite{Kadison_Kastler, Ch_1974, Ch_1977, Ch_1980, Ch_et_al_2010}). 

Another natural way to study the relative position of two masas is through the notion of the interior angle between them (\cite{Bakshi_gupta_2021}). While the Kadison–Kastler distance is a metric notion of proximity, the interior angle provides a geometric perspective on how two masas are situated relative to one another.

In the finite-dimensional setting, and in particular for masas in $\mn{n}$, this question becomes especially concrete and permits an explicit analysis of the relative position of two masas. Such an analysis can reveal connections between metric notions of proximity and geometric or structural invariants associated with the pair. 
Our aim in this paper is to investigate these connections for masas in matrix algebras, with particular emphasis on the Kadison--Kastler distance and the interior angle between masas.

These two viewpoints turn out to be closely related in the simplest nontrivial finite-dimensional setting. Indeed, in $\mn{2}$, we establish the following:
\vspace{0.3cm}

\noindent \textbf{Theorem A:} (Corollary \ref{cor:dkk-is-sin-of-angle})  For any two intermediate subalgebras $\C\otimes \mathbb{I}_2 \subsetneq\sA, \sB\subsetneq \mn{2}$, we have
    \[
    \dkk(\sA, \sB) = \sin \alpha(\sA, \sB).
    \]

Thus, in dimension two, the Kadison--Kastler distance and the interior angle despite being completely different in nature encode the exactly same relative-position information. This relationship provides a natural bridge between the metric and geometric viewpoints and forms a central theme of our study.

We first undertake a detailed analysis of the Kadison--Kastler distance between masas in $\mn{2}$. For arbitrary unitaries $u,v\in \mn{2}$, we obtain an explicit formula for $\dkk(u\Delta u^*,v\Delta v^*)$, in particular we prove the following:
\vspace{0.3cm}

\noindent \textbf{Theorem B:} (Corollary \ref{cor:d_KK_between_masa})
For any  two $2\times 2$ unitary matrices $u, v$, we have
\[
\dkk(u\Delta u^*, v\Delta v^*)
=
2\left|\bar{u}_{11}v_{11}+\bar{u}_{21}v_{21}\right|
\left|\bar{u}_{11}v_{12}+\bar{u}_{21}v_{22}\right|.
\]

From this theorem, it immediately follows that the range of this distance is precisely the interval $[0,1]$ (see Corollary \ref{cor:d_KK_range}). We further show that the maximal value $1$ is attained exactly when the relative unitary $u^*v$ is a Hadamard unitary (see Theorem \ref{thm:hadamard-uni}). In addition, the relation between the Kadison--Kastler distance and the interior angle established above gives a geometric interpretation of this extremal case. In fact, we have characterize those $2\times 2$ unitary matrices for which this distance is exactly $r$:
\vspace{0.3cm}

\noindent \textbf{Theorem C:} (Theorem \ref{thm:dkk-r-char})
    Let $0\le r \le 1$ and let $u$ be a $2\times 2$ unitary matrix. Then
    \[
    \dkk (\Delta, u\Delta u^*) = r
    \iff
    |u_{ij}| = \frac{1}{\sqrt{1+r}\pm \sqrt{1-r}}
    \text{ for all } i, j.
    \]

This yields a decomposition of the $2\times 2$ unitary group $\sU(\mn{2})$ according to the Kadison--Kastler distance from the diagonal masa $\Delta$ (see Remark \ref{rem:u-decom-dkk}):
\[
\sU(\mn{2}) = \bigsqcup_{0\le r\le 1}
\left\{ u \in \sU(\mn{2}) : \dkk (\Delta, u\Delta u^*) = r \right\}.
\]

 We also obtain a rigidity result for the groupoid normaliser of $\Delta$: for $v\in \mn{2}$ belonging to the groupoid normaliser of $\Delta$, the distance between $\Delta$ and $v\Delta v^*$ can take only its minimal or maximal value (see Corollary \ref{cor:groupoid_d_KK}). These results provide a complete description of the metric behavior of masas in $\mn{2}$ and identify Hadamard unitaries as the configurations corresponding to maximal separation.

The preceding analysis also reveals a remarkable connection with other invariants associated with pairs of masas. In particular, we investigate the relationship between the Kadison--Kastler distance, the interior angle, the Connes--St\o rmer modified entropy (\cite{Choda2008}), and the Popa's probabilistic index (\cite{Pimsner_Popa_1986}). Although these quantities arise from different perspectives, in $\mn{2}$ they detect exactly the same extremal configuration. More precisely, we prove the following:
\vspace{0.3cm}

\noindent \textbf{Theorem D:} (Theorem \ref{thm:equivalence})
For any  two $2\times 2$ unitary matrices $u, v$. The following are equivalent:
\begin{itemize}
    
     \item[(i)] $\dkk\left(u\Delta u^*,v\Delta v^*\right)=1$.
    
    \item[(ii)] $\alpha\left(u\Delta u^*,v\Delta v^*\right)=\frac{\pi}{2}$.

    \item[(iii)] $h(u\Delta u^* \mid v\Delta v^*)=\log(2)$.
    
    \item[(iv)] $\lambda(u\Delta u^*,v\Delta v^*)=\frac{1}{2}$.

    \item[(v)] $u^*v$ is a Hadamard unitary.
\end{itemize}

Thus, in $\mn{2}$, four apparently different ways of measuring the relative position of two masas -- metric, geometric, entropic, and probabilistic -- are equivalent at the extremal level. This provides a direct link between the geometry of masas and several fundamental invariants arising in operator-algebraic and subfactor theory.

The preceding results reveal a particularly rigid picture for masas in $\mn{2}$, in which the Kadison--Kastler distance, interior angle, Connes--St\o rmer modified entropy, and Popa's probabilistic index are closely intertwined. We next investigate which aspects of this picture persist in higher dimensions. While the precise identity between the Kadison--Kastler distance and the interior angle established above is specific to the two-dimensional setting, the interior angle itself admits a natural and explicit description for masas in arbitrary matrix algebras. To this end, let $\Delta^{(n)}$ denote the diagonal masa in $\mn{n}$. We study the interior angle between masas of the form $u\Delta^{(n)}u^*$ and $v\Delta^{(n)}v^*$ for arbitrary unitaries $u,v\in\mn{n}$, thereby extending the geometric aspect of our analysis to general finite dimensions.

For arbitrary $n$, we first establish a unitary invariance property of the interior angle:
\[
\alpha\left(u\Delta^{(n)}u^*,v\Delta^{(n)}v^*\right)
=
\alpha\left(\Delta^{(n)},u^*v\Delta^{(n)}v^*u\right).
\]
for all unitaries $u,v\in\mn{n}$ (see Theorem \ref{thm:uni-inv-angle}). This reduces the computation of the angle between arbitrary masas to that between the diagonal masa and a unitary conjugate of it. Using this invariance, we derive an explicit formula for the interior angle in terms of the entries of the relative unitary $u^*v$. In particuar, we prove the following:
\vspace{0.3cm}

\noindent \textbf{Theorem E:} (Theorem \ref{thm:angle-betwn-masas})
 Let $u,v\in\mn{n}$ be $n \times n$ unitary matrices. Then
\[
\cos\alpha\left(u\Delta^{(n)}u^*,v\Delta^{(n)}v^*\right)
=
\frac{n}{n-1}
\left\|
\left[
|u^*v|^{\circ 2}\circ(u^*v)
-\frac{1}{n}u^*v
\right]
\right\|_{\mathrm{op}}.
\]
where $\circ$ denotes the componentwise Hadamard product. As an immediate consequence, we obtain a characterization of the maximal angle:
\vspace{0.3cm}

\noindent \textbf{Theorem F:} (Theorem \ref{thm:angle-hadamard-comm-sq})
Let $u,v\in\mn{n}$ be $n\times n$ unitary matrices. Then
\[
\alpha(u\Delta^{(n)}u^*,v\Delta^{(n)}v^*)=\frac{\pi}{2}
\quad\Longleftrightarrow\quad
u^*v\ \text{is a Hadamard unitary}.
\]
Thus, the same Hadamard condition that characterizes maximal separation in the $\mn{2}$ case continues to govern the extremal interior angle in arbitrary dimension. Furthermore, rather unexpectedly, in Remark~\ref{rem:angle-not-good}, we show that the interior angle $\alpha(\Delta^{(n)},u\Delta^{(n)}u^*)$ can vanish even when $\Delta^{(n)}\neq u\Delta^{(n)}u^*$ for $n\geq 3$. This phenomenon persists even when the algebras are required to be simple. In particular, Corollary~\ref{cor:angle-not-inv-sim} provides distinct simple intermediate $C^*$-subalgebras with vanishing interior angle, thereby answering in the negative the question raised in \cite[Remark 3.4]{gupta2024} concerning whether a vanishing interior angle necessarily implies equality without the irreducibility assumption. Consequently, in general, the interior angle is not an invariant of the underlying $C^*$-algebras. This contrasts with the situation for $\mathrm{II}_1$-subfactors (\cite[Proposition 2.3]{bakshi-etal-2019}), irreducible inclusions of simple $C^*$-algebras (\cite[Proposition 5.10-(i)]{Bakshi_gupta_2021}), and intermediate subalgebras of $\mn{2}$ (\cite[Corollary 4.5-(2)]{gupta2024}), where the vanishing of the interior angle implies equality of the corresponding subalgebras. Finally, we show that the interior angle between masas in $\mn{n}$ attains every value in the interval $[0,\pi/2]$ (see Theorem~\ref{thm:range_angle}).

Finally, in $\mn{n}$ for $n \ge 3$, the situation changes drastically. For instance, although a Hadamard unitary $u \in \mn{n}$ maximizes the distance between two masas—that is, $\dkk( \Deltan, u \Deltan u^*) = 1$ (see Theorem \ref{thm:hadamard-makes-d-1})—for $n \ge 3$ there always exists a non-Hadamard unitary that can also maximize this distance (see Remark \ref{rem:dkk-max-by-non-hada}). Thus, the equivalence between the maximum Kadison-Kastler distance and Hadamard unitaries observed in Theorem \ref{thm:equivalence} for $n = 2$ is no longer true for $n \ge 3$. Furthermore, in Remark \ref{rem:dkk-is-not-sin-dim-n} we observe that the relationship expressing $\dkk$ as the sine of the interior angle $\alpha$ in the $n = 2$ case (established in Theorem \ref{thm:dkk-is-sin}) does not hold in every dimension $n \ge 3$.

We expect that the geometric viewpoint developed in this paper may have applications in finite-dimensional quantum information theory. The appearance of Hadamard unitaries as precisely the unitaries realizing the maximal Kadison--Kastler distance and the maximal interior angle between masas suggests a connection with the theory of mutually unbiased bases and complementary observables, where Hadamard transformations play a fundamental role. From this perspective, the Kadison--Kastler distance and the interior angle may provide quantitative measures of the relative position or incompatibility of classical observables represented by masas. We hope that the relationships established here, particularly the characterization of extremal configurations in terms of Hadamard unitaries, can be developed further to investigate connections between operator-algebraic notions of perturbation and geometric separation and quantitative notions of distinguishability, incompatibility, and complementarity in quantum information theory.

The paper is organized as follows. In \S~\ref{sec:prelims}, we recall the relevant definitions and preliminary results concerning the Kadison--Kastler distance, interior angles, Connes--St\o rmer modified entropy, and Popa's probabilistic index. In \S~\ref{sec:m2}, we study masas in $\mn{2}$ and obtain an explicit formula for their Kadison--Kastler distance, showing in particular that its range is $[0,1]$ and characterizing the cases of maximal distance. In \S~\ref{sec:relationship}, we establish the relationship between the Kadison--Kastler distance and the interior angle and investigate the connections with the Connes--St\o rmer modified entropy and Popa's probabilistic index. In \S~\ref{sec:dis-in-ndim}, we study the Kadison--Kastler distance in dimensions $n \ge 3$. First, we derive an explicit formula for the Kadison--Kastler distance between masas arising from rotation matrices. Using this formula, we prove that the distance between masas in $\mn{n}$ attains every value in $[0,1]$ for all $n \ge 3$. Furthermore, in contrast to the case $n=2$, we demonstrate that the maximal distance between pairs of masas can also be attained by non-Hadamard unitaries. In \S~\ref{sec:mn}, we turn to masas in $\mn{n}$ and study their interior angle. We establish unitary invariance, derive an explicit formula for the angle, characterize the maximal angle in terms of Hadamard unitaries, and prove that every value in $[0,\pi/2]$ is attained.

    
    \section{Preliminaries}
    \label{sec:prelims}
    The purpose of this section is to introduce the notation and conventions that will be used throughout the paper. We write $\mn{n}$ for the algebra of $n\times n$ complex matrices, and denote by $\sU(\mn{n})$ the unitary group of $\mn{n}$. The diagonal masa in $\mn{n}$ is denoted by
\[
\Delta^{(n)}:= \left\{
\begin{pmatrix}
\lambda_1 & 0 & \cdots & 0 \\
0 & \lambda_2 & \cdots & 0 \\
\vdots & \vdots & \ddots & \vdots \\
0 & 0 & \cdots & \lambda_n
\end{pmatrix}
: \lambda_1,\ldots,\lambda_n \in \mathbb{C}
\right\}.
\]
In the $2\times 2$ case, we shall simply write $\Delta$ for $\Delta^{(2)}$ throughout.

For an integer $n \ge 2$, angle $\theta \in \mathbb{R}$, and pair of indices $1 \le j < k \le n$, we denote by $R_{jk}(\theta) \in \sU(\mn{n})$ the principal rotation matrix in the $j$-$k$ plane. It is defined as the unitary matrix obtained from the $n \times n$ identity matrix by replacing the $2 \times 2$ principal submatrix at the intersection of the $j$-th and $k$-th rows and columns with the standard two-dimensional rotation block:
\[
r_\theta :=
\begin{pmatrix}
\cos\theta & -\sin\theta\\
\sin\theta & \cos\theta
\end{pmatrix}
\in \sU(\mn{2}).
\]
Explicitly, this means the restricted $j$-$k$ block is given by $[R_{jk}(\theta)]_{j,j} = [R_{jk}(\theta)]_{k,k} = \cos\theta$, $[R_{jk}(\theta)]_{j,k} = -\sin\theta$, and $[R_{jk}(\theta)]_{k,j} = \sin\theta$, while all remaining diagonal entries are $1$ and all other off-diagonal entries are $0$.


For matrices $A \in \mn{m \times n}$ and $B \in \mn{l \times k}$, we define their Kronecker product $A\otimes B \in \mn{ml \times nk}$ by
\[
A\otimes B:= [B_{ij}A]_{1\le i\le l,\,1\le j\le k}.
\]
Notice that, for any $A \in \mn{n}$ and $p \in \N$, we have
\[
A\otimes \mathbb{I}_p
=
\mathrm{diag}(\underbrace{A, \dots, A}_{p\text{ copies}})\in \mn{pn}.
\]
Note that for two $n\times n$ matrices $A, B\in \mn{n}$, their Kronecker product $A\otimes B$ belongs to $\mn{n^2}$. Moreover, the set of such Kronecker products spans the entire algebra $\mn{n^2}$. Thus, throughout this paper, we use the following definition of the tensor product algebra:
\[
\mn{n} \otimes \mn{n}:= \spn \left\{ A\otimes B: A, B \in \mn{n} \right\}=\mn{n^2}.
\]

Let $X \in \mn{n}$ with column partitioning $X = \begin{pmatrix} x_1 & x_2 & \dots & x_n \end{pmatrix}$, where each $x_j \in \mathbb{C}^n$ is the $j$-th column vector. We define the following linear isomorphism $\operatorname{vec}: \mn{n} \to \C^{n^2}$ by:
$$
\operatorname{vec}(X) := \begin{pmatrix} x_1 \\ x_2 \\ \vdots \\ x_n \end{pmatrix} = \begin{pmatrix} X_{11} \\ X_{21} \\ \vdots \\ X_{n1} \\ X_{12} \\ \vdots \\ X_{n2} \\ \vdots \\ X_{nn} \end{pmatrix} \in \mathbb{C}^{n^2}
$$

We use $\{E_{ij}\}_{i,j=1}^n$ to denote the standard matrix units in $\mn{n}$, while $\{e_1,\dots,e_n\}$ denotes the standard orthonormal basis (ONB) of the Hilbert space $\C^n$. Note that
\[
\operatorname{vec}(E_{ij}) = e_i \otimes e_j.
\]

Finally, throughout this paper, we use the notation $\sB \subseteq^E \sA$ to indicate that $E:\sA\to\sB$ is a conditional expectation from the $C^*$-algebra $\sA$ onto $\sB$.

\subsection{Kadison--Kastler Distance}

The perturbation theory of operator algebras on Hilbert spaces was initiated by R.~V.~Kadison and D.~Kastler in 1972 \cite{Kadison_Kastler}. In their work, they introduced a notion of distance between operator algebras acting on a Hilbert space $\sH$. This notion, however, is more generally defined for subspaces of an arbitrary normed space. We begin by recalling the definition of the Kadison--Kastler distance for subspaces of normed spaces and then record some standard facts that will be used in the setting of subalgebras of $\Bh$.

For any normed space $\sX$, we denote its closed unit ball by $\ball{\sX}$. For a subset $\sS$ of $\sX$ and an element $x\in\sX$, the distance from $x$ to $\sS$ is defined by
\[
d(x,\sS):=\inf\{\|x-s\|:s\in\sS\}.
\]

For any two subspaces $\sY$ and $\sZ$ of a normed space $\sX$, the Kadison--Kastler distance between them, denoted by $\dkk(\sY,\sZ)$, is defined as the Hausdorff distance between their closed unit balls.

\begin{definition}[{\cite[Definition A]{Kadison_Kastler}}]
    \[
    \dkk(\sY,\sZ)
    :=\max\left\{
    \sup_{y\in\ball{\sY}}d(y,\ball{\sZ}),
    \sup_{z\in\ball{\sZ}}d(z,\ball{\sY})
    \right\}.
    \]
\end{definition}

\begin{remark}\label{KK-facts}
  Let $\sX$ be a normed space. The following facts about the Kadison--Kastler distance are well known:
  \begin{enumerate}
    \item[(i)] $0\le\dkk(\sY,\sZ)\le1$ for all subspaces $\sY,\sZ$ of $\sX$.
    \item[(ii)] $\dkk$ is a metric on the collection of all subspaces of $\sX$ (see \cite[\S~I.4]{viro-topo}).
  \end{enumerate}
\end{remark}

\noindent For further details, see \cite{Kadison_Kastler} or \cite{Ch_et_al_2010}.

\subsection{Watatani's Reduced $C^*$-Basic Construction}\label{index-basic-construction}

Let $\sB \subseteq \sA$ be an inclusion of unital $C^*$-algebras with a common identity, and let $E:\sA\to\sB$ be a faithful conditional expectation. Then $\sA$ becomes a right pre-Hilbert $C^*$-module over $\sB$ with respect to the canonical $\sB$-valued inner product
\[
\langle x,y\rangle_{\sB}:=E(x^*y), \qquad x,y\in\sA.
\]
We denote by $\mathcal{E}$ the Hilbert $C^*$-module obtained by completing $\sA$ with respect to the norm induced by this inner product.

To distinguish elements of the $C^*$-algebra $\sA$ from elements of the pre-Hilbert $C^*$-module $\sA$, following \cite{watatani1990}, we consider the canonical inclusion map
\[
\eta:\sA\to\mathcal{E}.
\]
Thus,
\begin{equation}\label{eta-norm}
    \|\eta(x)\|_\sB:=\|E(x^*x)\|^{1/2}\leq\|x\|
\end{equation}
for all $x\in\sA$.

Let $\mathcal{L}_{\sB}(\mathcal{E})$ denote the unital $C^*$-algebra of adjointable operators on $\mathcal{E}$. Every element of $\mathcal{L}_{\sB}(\mathcal{E})$ is, in particular, a $\sB$-module map on $\mathcal{E}$. There is a natural $C^*$-embedding
\[
\lambda:\sA\to\mathcal{L}_{\sB}(\mathcal{E})
\]
given by
\[
\lambda(a)\eta(x)=\eta(ax), \qquad a,x\in\sA.
\]
We henceforth identify $\sA$ with its image $\lambda(\sA)$ in $\mathcal{L}_{\sB}(\mathcal{E})$.

There is also a natural projection
\[
e_1\in\lambda(\sB)'\cap\mathcal{L}_{\sB}(\mathcal{E})
\]
which, under the above identification, satisfies
\[
e_1(\eta(x))=\eta(E(x))
\quad\text{and}\quad
e_1xe_1=E(x)e_1
\]
for all $x\in\sA$. The projection $e_1$ is called the Jones projection associated with the conditional expectation $E$ and is also commonly denoted by $e_\sB$.

Consequently,
\[
\overline{\operatorname{span}}\{xe_1y:x,y\in\sA\}
\]
is a $C^*$-subalgebra of $\mathcal{L}_{\sB}(\mathcal{E})$. This $C^*$-algebra is called Watatani's reduced $C^*$-basic construction associated with the inclusion $\sB\subseteq\sA$ and the conditional expectation $E$, and is denoted by $\sA_1$. Thus,
\[
\sA_1:=\overline{\operatorname{span}}\{xe_1y:x,y\in\sA\}
\subseteq\mathcal{L}_{\sB}(\mathcal{E}).
\]
For further details on Watatani's basic construction, we refer the reader to \cite{watatani1990}.

\subsection{Interior Angle Between Subalgebras}

Let $\sB\subseteq\sA$ be an inclusion of unital $C^*$-algebras. A conditional expectation $E:\sA\to\sB$ is said to have finite Watatani index if there exists a finite set $\{\lambda_1,\dots,\lambda_n\}\subseteq\sA$ such that
\[
x=\sum_{i=1}^n \lambda_i E(\lambda_i^*x)
\qquad\text{for all }x\in\sA.
\]
Such a set $\{\lambda_i\}_i$ is called a \emph{quasi-basis} for $E$ \cite{watatani1990}.

Let $\operatorname{IMS}(\sB,\sA,E)$ denote the collection of all intermediate $C^*$-subalgebras $\sC$ satisfying $\sB\subseteq\sC\subseteq\sA$ for which there exists a conditional expectation $F:\sA\to\sC$ such that
\[
E_{\restriction_\sC}\circ F=E.
\]
For example, for any unitary $u \in \mn{n}$, we have
\[
u\Delta^{(n)}u^* \in \mathrm{IMS}\left(
\C \otimes \mathbb{I}_n,\, \mn{n},\, \frac{1}{n}\operatorname{Tr}(-) \otimes \mathbb{I}_n
\right).
\]

For $\sC,\sD\in\operatorname{IMS}(\sB,\sA,E)$, we denote by $e_\sC$ and $e_\sD$ the corresponding Jones projections in $\sC_1$ and $\sD_1$, respectively. We will use the following result.

\begin{prop}[{\cite[Proposition 2.7, Remark 2.3]{gupta2024}}]\label{prop:gupta-jones-proj}
    Let $\sB\subseteq\sA$ be an inclusion of unital $C^*$-algebras, let $E:\sA\to\sB$ be a finite-index conditional expectation, and let $\sC\in\operatorname{IMS}(\sB,\sA,E)$. Then:
    \begin{itemize}
        \item[(i)] $\sL_\sC(\sA)\subseteq\sL_\sB(\sA)$.
        \item[(ii)] $\sC_1\subseteq\sA_1$.
        \item[(iii)] $e_\sC e_\sB=e_\sB=e_\sB e_\sC$.
        \item[(iv)] The two norms $\| \cdot \|_{\sB}$ and $\|\cdot \|$ on $\sA$ are equivalent. In particular, $\sA$ itself is a Hilbert $\sB$-module.
        \item[(v)] $\sA_1$ is unital and  $\sA_{1}=C^*(\sA, e_\sB)=\sL_\sB(\sA)$.
        \item[(vi)] there exists a finite index conditional expectation $\Tilde{E}:\sA_1 \to \sA$ (called \emph{dual conditional expectation}) with a quasibasis $\{ \lambda_i e_\sB\operatorname{Ind}(E)^{1/2}\}$ that satisfies the equation 
        \[
        \Tilde{E}(x e_\sB y) = \operatorname{Ind}(E)^{-1} xy
        \]
        for $x, y\in \sA$.
    \end{itemize}
\end{prop}

We now recall the notion of the interior angle between intermediate $C^*$-subalgebras.

\begin{definition}(\cite[Definition 5.1]{Bakshi_gupta_2021})\label{def:interior_angle}
    Let $\sB\subseteq\sA$ be an inclusion of unital $C^*$-algebras with a finite-index conditional expectation $E:\sA\to\sB$. For $\sC,\sD\in\operatorname{IMS}(\sB,\sA,E)$, the \emph{interior angle} between $\sC$ and $\sD$ with respect to $E$, denoted by $\alpha(\sC,\sD)$, is defined by
    \[
    \cos\left(\alpha(\sC,\sD)\right)
    =
    \frac{
    \left\|\langle e_{\sC}-e_{\sB},e_{\sD}-e_{\sB}\rangle_{\sA}\right\|
    }{
    \left\|e_{\sC}-e_{\sB}\right\|_{\sA}
    \left\|e_{\sD}-e_{\sB}\right\|_{\sA}
    }.
    \]
    We take $\alpha(\sC,\sD)$ to be the unique value in the interval $[0,\pi/2]$ satisfying the above equation.
\end{definition}

For further details on the interior angle between intermediate $C^*$-subalgebras, we refer the reader to \cite{Bakshi_gupta_2021}.

\subsection{Pimsner--Popa Probabilistic Index and Connes--St\o rmer Modified Entropy}

We conclude the preliminaries by recalling two further notions that will be used in \S~\ref{sec:m2}. The first is the Pimsner--Popa probabilistic index.

\begin{definition}[{\cite[Definition 2.12]{BakshiGuin2025}}]
    Let $\sM$ be a finite von Neumann algebra, and let $\sP,\sQ\subseteq\sM$ be von Neumann subalgebras. The Pimsner--Popa probabilistic index of the inclusion $\sQ\subseteq\sP$ is defined to be $\lambda(\sP,\sQ)^{-1}$, where
    \[
    \lambda(\sP,\sQ):=
    \sup\left\{
    t\geq 0:E_\sQ(x)\geq tx\text{ for all }x\in\sP_+
    \right\},
    \]
    and $E_\sQ$ denotes the trace-preserving conditional expectation from $\sM$ onto $\sQ$.
\end{definition}

The second notion we recall is the Connes--St\o rmer modified relative entropy between two von Neumann subalgebras of a finite von Neumann algebra. We use the following definition.

\begin{definition}[\cite{Choda2008, Choda2011}]
    Let $(\sM,\tau)$ be a finite von Neumann algebra, and let $\sP,\sQ\subseteq\sM$ be von Neumann subalgebras. Let $E_\sQ:\sM\to\sQ$ denote the trace-preserving conditional expectation. For $n\in\N$, let
    \[
    \gamma=\left\{x_1,\dots,x_n\in\sP_+:\sum_{j=1}^n x_j=1_\sM\right\}
    \]
    be a partition of $1_\sM$. Let $\eta:[0,\infty)\to\R$ be the continuous function defined by
    \[
    \eta(t):=-t\log t.
    \]
    For such a partition $\gamma$, define
    \[
    h_\gamma(\sP\mid\sQ)
    :=
    \sum_{j=1}^n
    \left(
    \tau\circ\eta\left(E_\sQ(x_j)\right)
    -
    \tau\circ\eta\left(x_j\right)
    \right).
    \]
    Then
    \[
    h(\sP\mid\sQ):=\sup_\gamma h_\gamma(\sP\mid\sQ)
    \]
    is called the \emph{Connes--St\o rmer modified relative entropy} between $\sP$ and $\sQ$.
\end{definition}

    \section{The Kadison-Kastler Distance in $\mn{2}$}
    \label{sec:m2}
    One of the main purposes of this section is to obtain an explicit formula for the quantity
\[
\dkk(u\Delta u^*,v\Delta v^*)
\]
for arbitrary $2\times2$ unitary matrices $u,v\in\mn{2}$. We begin with a few preliminary lemmas and first compute the Kadison--Kastler distance between $\Delta$ and $u\Delta u^*$ for a $2\times2$ unitary $u$ (Theorem \ref{thm:dkk-2-case}). As a consequence, we show that, in $\mn{2}$, the Kadison--Kastler distance can attain every value in the interval $[0,1]$ (Corollary \ref{cor:d_KK_range}). Moreover, for each $0\le r\le1$, Theorem \ref{thm:dkk-r-char} characterizes all $2\times2$ unitary matrices $u$ satisfying
\[
\dkk(\Delta,u\Delta u^*)=r.
\]

The case $r=1$ is particularly interesting. We show in Theorem \ref{thm:hadamard-uni} that
\[
\dkk(\Delta,u\Delta u^*)=1
\]
if and only if $u$ is a Hadamard unitary (see Definition \ref{def:Hadamard_uni}). Consequently, we obtain a decomposition of the $2\times2$ unitary group $\sU(\mn{2})$ according to the value of the Kadison--Kastler distance (see Remark \ref{rem:u-decom-dkk}):
\[
    \sU(\mn{2}) = \bigsqcup_{0\le r\le1}
    \left\{u\in\sU(\mn{2}):\dkk(\Delta,u\Delta u^*)=r\right\}.
\]

We then establish a unitary invariance property of the Kadison--Kastler distance for general $n$ (see Theorem \ref{prop:dkk-gen-masa}):
\[
\dkk\left(u\Delta^{(n)}u^*,v\Delta^{(n)}v^*\right)
=
\dkk\left(\Delta^{(n)},u^*v\Delta^{(n)}v^*u\right).
\]
This identity is the key ingredient in computing the desired distance $\dkk(u\Delta u^*,v\Delta v^*)$ for unitaries $u,v\in\mn{2}$ (see Corollary \ref{cor:d_KK_between_masa}). Using this formula, we also obtain an explicit characterization of the unitary normalizer
\[
\unor{\mn{2}}(\Delta):=
\{u\in\sU(\mn{2}):u\Delta u^*=\Delta\}
\]
(see Corollary \ref{cor:unitary_nor_char}). On the other hand, for an element $v$ of the \emph{groupoid normalizer} 
\[\grnor{\mn{2}}(\Delta):=\{ v \in \mn{2}: v\Delta v^*\cup v^*\Delta v \subseteq \Delta \text{ and }vv^*v = v \}\]
we show that the only possible values of $\dkk(\Delta,v\Delta v^*)$ are $0$ and $1$ (see Corollary \ref{cor:groupoid_d_KK}).

We conclude this section with our second main goal in dimension $2$, namely, investigating the relationship between the Kadison--Kastler distance and the interior angle. We prove (see Corollary \ref{cor:relation_angle_distance}) that, for every $2\times2$ unitary $u$,
\[
\dkk(\Delta,u\Delta u^*)
=
\sin\alpha(\Delta,u\Delta u^*).
\]
Thus, two quantities that are a priori unrelated turn out to encode essentially the same information for masas in the $2\times2$ setting.

\begin{lem}\label{lem:min-calc}
    For $t \in [-1, 1]$, we have
\[
\min_{\alpha \in \overline{\mathbb{D}}} \left( 1 + |\alpha|^2 - 2 \Re(\alpha) t \right)= 1-t^2.
\]
\end{lem}

\begin{lem}\label{lem:distance_cal}
    Let $u=(u_{ij})\in \sU(\mn{2})$ be a $2\times 2$ unitary. Then we have
\begin{enumerate}
    \item[(i)] $\displaystyle \sup_{x \in \ball{\Delta}} \operatorname{dist} (x, \ball{u\Delta u^*}) = 2 |u_{11} u_{12}|.$
    \item[(ii)] $\displaystyle \sup_{y \in \ball{u\Delta u^*}} \operatorname{dist}(y, \ball{\Delta})=2|u_{11} u_{12}|.$
\end{enumerate}
\end{lem}
\begin{proof}
\noindent (i) Any element $x \in \ball{\Delta}$ is a diagonal matrix bounded by $1$ in operator norm:$$x = \begin{pmatrix} \lambda & 0 \\ 0 & \mu \end{pmatrix} \quad \text{with} \quad \max(\vert{}\lambda\vert{}, \vert{}\mu\vert{}) \le 1$$We can express $x$ using the identity matrix $\mathbb{I}_{2}$ and the Pauli $z$ matrix $\sigma_z$:
$$x = \frac{\lambda+\mu}{2} \mathbb{I}_{2} + \frac{\lambda-\mu}{2} \sigma_z$$

The set $\ball{u\Delta u^*}$ consists of matrices of the form $y = u \text{diag}(\alpha, \beta) u^*$ where $\max(\vert{}\alpha\vert{}, \vert{}\beta\vert{}) \le 1$. We can similarly decompose any $y \in \ball{u\Delta u^*}$ as:
$$y = \frac{\alpha+\beta}{2} \mathbb{I}_{2} + \frac{\alpha-\beta}{2} y' \quad \text{where} \quad y' = u \sigma_z u^*$$    

Using the orthogonality of the rows of the unitary matrix $u$ (specifically $u_{11}\bar{u}_{21} + u_{12}\bar{u}_{22} = 0$), we can compute the exact form of the matrix $y'$:$$y' = \begin{pmatrix} \vert{}u_{11}\vert{}^2 - \vert{}u_{12}\vert{}^2 & 2u_{11}\bar{u}_{21} \\ 2u_{21}\bar{u}_{11} & \vert{}u_{21}\vert{}^2 - \vert{}u_{22}\vert{}^2 \end{pmatrix} = \begin{pmatrix} t & w \\ \bar{w} & -t \end{pmatrix}$$where $t = \vert{}u_{11}\vert{}^2 - \vert{}u_{12}\vert{}^2 \in [-1, 1]$. Because $y'$ is a unitary reflection, $t^2 + \vert{}w\vert{}^2 = 1$, meaning $1 - t^2 = \vert{}w\vert{}^2 = 4\vert{}u_{11}\vert{}^2\vert{}u_{12}\vert{}^2$.

\noindent\textbf{Claim.} $\sup_{x \in \ball{\Delta}} \operatorname{dist} (x, \ball{u\Delta u^*}) \le \sqrt{1 - t^2}$.

We want to show that for any $x \in \ball{\Delta}$, there exists some $y \in \ball{u\Delta u^*}$ such that $\Vert{}x - y\Vert{} \le \sqrt{1-t^2}$.

Let $x = \text{diag}(\lambda, \mu)\in \ball{\Delta}$. Choose $y$ by setting its eigenvalues $\alpha$ and $\beta$ as follows:$$\alpha = \lambda \frac{1+t}{2} + \mu \frac{1-t}{2}$$$$\beta = \lambda \frac{1-t}{2} + \mu \frac{1+t}{2}$$Because $t \in [-1, 1]$, both $\alpha$ and $\beta$ are convex combinations of $\lambda$ and $\mu$. Since $\lambda, \mu$ belong to the closed unit disk, so do $\alpha$ and $\beta$, meaning our chosen $y$ is inside $\ball{u\Delta u^*}$.

Hence, 
$$\frac{\alpha+\beta}{2} = \frac{\lambda+\mu}{2}\text{ and }\quad\frac{\alpha-\beta}{2} = \frac{\lambda-\mu}{2} t$$
Therefore, $$x - y = \frac{\lambda-\mu}{2} (\sigma_z - t y')$$

The matrix $(\sigma_z - t y')$ evaluates to
$$\sigma_z - t y' = \begin{pmatrix} 1-t^2 & -tw \\ -t\bar{w} & -(1-t^2) \end{pmatrix}$$

This is a trace-zero self-adjoint matrix. Its norm is the square root of its determinant's absolute value, i.e.,
$$\Vert{}\sigma_z - t y'\Vert{} = \sqrt{(1-t^2)^2 + t^2\vert{}w\vert{}^2} = \sqrt{(1-t^2)^2 + t^2(1-t^2)} = \sqrt{1-t^2}$$

Therefore, the distance from our arbitrary $x$ to $\ball{u\Delta u^*}$ is bounded by,
$$\Vert{}x - y\Vert{} = \left\vert{} \frac{\lambda-\mu}{2} \right\vert{} \sqrt{1-t^2}.$$
Since $\lambda, \mu$ are in the unit disk, the maximum value of $\vert{}\frac{\lambda-\mu}{2}\vert{}$ is $1$. Thus, $\text{dist}(x, \ball{u\Delta u^*}) \le \sqrt{1-t^2}$.

\noindent\textbf{Claim.} $\sup_{x \in \ball{\Delta}} \operatorname{dist} (x, \ball{u\Delta u^*}) \ge \sqrt{1 - t^2}$.

For this consider  $x = \sigma_z\in \ball{\Delta}$ (where $\lambda = 1, \mu = -1$). We will show that for every $y \in B_{u\Delta u^*}$, the distance $\Vert{} \sigma_z - y \Vert{}$ is at least $\sqrt{1-t^2}$.

Let $y = u \operatorname{diag}(\alpha, \beta) u^* \in \ball{u\Delta u^*}$. Then $v_1 = (u_{11}, u_{21})^T$ is an eigenvector of $y$ with eigenvalue $\alpha \in \overline{\mathbb{D}}$.

Hence, 
$$\Vert{} (\sigma_z - y) v_1 \Vert{}^2 = \Vert{} \sigma_z v_1 - \alpha v_1 \Vert{}^2 = \left\Vert{} \begin{pmatrix} (1-\alpha)u_{11} \\ -(1+\alpha)u_{21} \end{pmatrix} \right\Vert{}^2 = \vert{}1-\alpha\vert{}^2 \vert{}u_{11}\vert{}^2 + \vert{}1+\alpha\vert{}^2 \vert{}u_{21}\vert{}^2
$$
Using $\vert{}u_{11}\vert{}^2 = \frac{1+t}{2}$ and $\vert{}u_{21}\vert{}^2 = \frac{1-t}{2}$, we get
$$\Vert{} (\sigma_z - y) v_1 \Vert{}^2= (1 - 2\Re(\alpha) + \vert{}\alpha\vert{}^2)\frac{1+t}{2} + (1 + 2\Re(\alpha) + \vert{}\alpha\vert{}^2)\frac{1-t}{2} = 1 + \vert{}\alpha\vert{}^2 - 2\Re(\alpha)t.$$
Therefore we have
\[
\|\sigma_z - y \|^2 \ge \Vert{} (\sigma_z - y) v_1 \Vert{}^2 = 1 + \vert{}\alpha\vert{}^2 - 2\Re(\alpha)t \ge \min_{\alpha \in \overline{\mathbb{D}}} \left( 1 + |\alpha|^2 - 2 \Re(\alpha) t \right)= 1-t^2,
\]
where the last equality follows from Lemma \ref{lem:min-calc}. Thus, we get
\[
\sup_{x \in \ball{\Delta}} \operatorname{dist} (x, \ball{u\Delta u^*}) \ge \operatorname{dist}(\sigma_z, \ball{u\Delta u^*})= \inf_{y \in \ball{u\Delta u^*}} \| \sigma_z - y \| \ge \sqrt{1 - t^2}.
\]

\noindent (ii) First note that any $y \in \ball{u \Delta u^*}$ is of the form $y = u x'u^*$ for some diagonal matrix $x'\in \ball{\Delta}$. Then the result follows from the previous lemma by noticing the fact that norm is invariant under unitary conjugation and 
    \[
    \operatorname{dist}(y, \ball{\Delta}) = \inf_{x \in \ball{\Delta}} \|y - x\| = \inf_{x \in \ball{\Delta}} \| ux'u^* - x\| = \inf_{x\in \ball{\Delta}}\|x' - u^* x u\|= \inf_{y \in \ball{u\Delta u^*}} \| x' - y \|.
    \]
\end{proof}

\begin{theorem}\label{thm:dkk-2-case}
   For every $2\times 2$ unitary matrix $u = (u_{ij})\in \sU(\mn{2})$, the Kadison--Kastler distance between $\Delta$ and its unitary conjugate $u\Delta u^*$ is given by
\[
\dkk(\Delta, u\Delta u^*) = 2 |u_{11}u_{12}|.
\]
\end{theorem}
\begin{proof}
    This follows directly from Lemma \ref{lem:distance_cal} and the definition of the Kadison--Kastler distance.
\end{proof}

\begin{cor}\label{cor:unitary_nor_char}
    A $2\times 2$ unitary $u\in \unor{\mn{2}}(\Delta)$ if and only if either $u\in\Delta$, or $u_{11}=0=u_{22}$ and
$|u_{12}|=1=|u_{21}|.
$
\end{cor}

It is well known that if $\sB \subsetneq \sC$ is a proper inclusion of subalgebras of $\Bh$, then $\dkk(\sB,\sC)=1$ (see \cite[Lemma 2.1]{Watatani_ino_2014}). Moreover, \cite{Gupta_Kumar_2024,Kumar_2026} provide non-trivial examples of subalgebras $\sA,\sB$ of $\Bh$ satisfying $\dkk(\sA,\sB)=1$. To the best of our knowledge, beyond these examples, there are no concrete examples demonstrating that the Kadison--Kastler distance between subalgebras can attain values strictly between $0$ and $1$. Using the preceding proposition, we now show that the Kadison--Kastler distance can, in fact, attain every value in the interval $[0,1]$.

\begin{cor}\label{cor:d_KK_range}
   For every $r \in [0,1]$, there exists a $2\times 2$ unitary matrix $u \in \sU(\mn{2})$ such that
\[
\dkk(\Delta, u\Delta u^*) = r.
\]
\end{cor}
\begin{proof}
   Take the $2\times 2$ rotation matrix corresponding to the angle $\theta=\frac{1}{2}\arcsin(r)$, namely,
\[
u=\begin{pmatrix}
\cos\theta & -\sin\theta\\
\sin\theta & \cos\theta
\end{pmatrix}.
\]
\end{proof}

We next characterize precisely those $2\times2$ unitary matrices for which the Kadison--Kastler distance is equal to a prescribed value $r\in [0, 1]$. The following result provides the desired characterization.

\begin{prop}\label{prop:hada-uni}
    Let $r$ be a real number with $0 \le r \le 1$. A $2\times 2$ matrix $u$ is unitary and satisfies $\vert{}u_{11}u_{12}\vert{}=r/2$ if and only if $u$ has the form
\[
u=
\begin{pmatrix}
a e^{i\phi_1} & b e^{i\phi_2}\\
b e^{i\phi_3} & -a e^{i(\phi_2+\phi_3-\phi_1)}
\end{pmatrix},
\]
where $\phi_1,\phi_2,\phi_3\in[0,2\pi)$, and the magnitudes $a,b$ are given by exactly one of the following two pairs:
\begin{enumerate}
    \item[(i)] $a = \frac{\sqrt{1+r}+\sqrt{1-r}}{2}$ and $b = \frac{\sqrt{1+r}-\sqrt{1-r}}{2}$;
    \item[(ii)] $a = \frac{\sqrt{1+r}-\sqrt{1-r}}{2}$ and $b = \frac{\sqrt{1+r}+\sqrt{1-r}}{2}$.
\end{enumerate}
\end{prop}

\begin{proof}
    \noindent($\implies$). Assume $u = \begin{pmatrix} u_{11} & u_{12} \\ u_{21} & u_{22} \end{pmatrix}$ is a unitary matrix such that $\vert{}u_{11} u_{12}\vert{} = r/2$. 
    
    Since the the rows and columns must form orthonormal bases, the entries of $u$ must satisfy:
    \begin{itemize}
        \item[1.] $\vert{}u_{11}\vert{}^2 + \vert{}u_{12}\vert{}^2 = 1$
        \item[2.] $\vert{}u_{11}\vert{}^2 + \vert{}u_{21}\vert{}^2 = 1 \implies \vert{}u_{21}\vert{} = \vert{}u_{12}\vert{}$
        \item[3.] $\vert{}u_{12}\vert{}^2 + \vert{}u_{22}\vert{}^2 = 1 \implies \vert{}u_{22}\vert{} = \vert{}u_{11}\vert{}$
    \end{itemize}
    Let $a = \vert{}u_{11}\vert{} = \vert{}u_{22}\vert{}$ and $b = \vert{}u_{12}\vert{} = \vert{}u_{21}\vert{}$. We are given $\vert{}u_{11} u_{12}\vert{} = r/2$, which translates to $ab = r/2$, or $2ab = r$. Therefore we get:
    $$(a+b)^2 = a^2 + b^2 + 2ab = 1 + r \implies a+b = \sqrt{1+r}$$
    $$(a-b)^2 = a^2 + b^2 - 2ab = 1 - r \implies a-b = \pm\sqrt{1-r}$$
    Adding and subtracting these two equations yields exactly the two pairs for $(a, b)$ stated in the theorem.

    Now we can express the entries of $u$ in polar form using our defined magnitudes $a$ and $b$ and arbitrary phases $\phi_1, \phi_2, \phi_3, \phi_4 \in [0, 2\pi)$:$$u = \begin{pmatrix} a e^{i\phi_1} & b e^{i\phi_2} \\ b e^{i\phi_3} & a e^{i\phi_4} \end{pmatrix}$$

    Unitarity also requires the rows to be orthogonal to each other:
    $$ab \left( e^{i(\phi_1-\phi_3)} + e^{i(\phi_2-\phi_4)} \right) =(a e^{i\phi_1})(b e^{-i\phi_3}) + (b e^{i\phi_2})(a e^{-i\phi_4})=u_{11} \bar{u}_{21} + u_{12} \bar{u}_{22} = 0$$
    If $r > 0$, then $ab \neq 0$, thus:
    $$e^{i(\phi_1-\phi_3)} = -e^{i(\phi_2-\phi_4)}$$

    Solving for $\phi_4$:$$\phi_4 \equiv \phi_2 + \phi_3 - \phi_1 - \pi \pmod{2\pi}$$

    Since $e^{-i\pi} = e^{i\pi} = -1$, substituting $\phi_4$ back into the expression for $u_{22}$ gives:$$u_{22} = a e^{i(\phi_2+\phi_3-\phi_1-\pi)} = -a e^{i(\phi_2+\phi_3-\phi_1)}$$

    \noindent ($\impliedby$). This follows immediately.
\end{proof}

From the preceding proposition, we immediately obtain the following characterization.

\begin{theorem}\label{thm:dkk-r-char}
    Let $0\le r \le 1$ and let $u$ be a $2\times 2$ unitary matrix. Then
    \[
    \dkk (\Delta, u\Delta u^*) = r
    \iff
    |u_{ij}| = \frac{1}{\sqrt{1+r}\pm \sqrt{1-r}}
    \text{ for all } i, j.
    \]
\end{theorem}

\begin{definition}[Hadamard Matrix]
A complex Hadamard matrix is a matrix $H \in \mn{n}$ whose all entries have the same modulus and which satisfies $HH^* = n\mathbb{I}_n$.
\end{definition}
\begin{definition}[Hadamard Unitary]\label{def:Hadamard_uni}
A complex $n \times n$ matrix $W$ is said to be a \emph{Hadamard unitary} if $\sqrt{n}W$ is a complex Hadamard matrix in $\mn{n}$.
\end{definition}

\begin{prop}\label{prop:hadamard-uni}
   The following conditions are equivalent:
\begin{itemize}
    \item[(i)] $W$ is a Hadamard unitary.
    \item[(ii)] $W\in \sU(\mn{n})$ and $|W_{ij}| = \frac{1}{\sqrt{n}}$ for all $i,j$.
\end{itemize}
\end{prop}

It is easy to see by Proposition \ref{prop:hadamard-uni}, the matrix $r_\theta$ is a Hadamard unitary if and only if $\theta = \frac{\pi}{4} + n\frac{\pi}{2}$, for all $n\in\mathbb{Z}.$

\begin{theorem}\label{thm:hadamard-uni}
    A unitary $u\in \mn{2}$ satisfies $\dkk(\Delta, u\Delta u^*)=1$ if and only if $u$ is a Hadamard unitary if and only if the following is a commuting square:
    \[
\begin{array}{ccc}
\Delta  & \subset & \mathbb{M}_2(\mathbb{C}) \\
\cup && \cup \\
\mathbb{C} & \subset & u\Delta u^*
\end{array}
\]
\end{theorem}
\begin{proof}
    The proof follows from Theorem \ref{thm:dkk-r-char} and \cite[\S 5.2.2]{JonesSundar97}.
\end{proof}

\begin{remark}\label{rem:u-decom-dkk}
    Define a function $\fK: \sU(\mn{2}) \to [0, 1]$ by
    \[
    \fK(u):= \dkk( \Delta, u\Delta u^*).
    \]
    By Corollary \ref{cor:cont-dkk}, the map $\fK$ is a continuous surjection and, in particular, a closed map. Hence, $\sU(\mn{2})$ admits the decomposition
    \[
    \sU(\mn{2}) = \bigsqcup_{0\le r\le 1} \fK^{-1} (r).
    \]
    Moreover, Proposition \ref{prop:hada-uni} gives
    \[
    \fK^{-1}(r) = \left\{ u \in \sU(\mn{2}): |u_{ij}| = \frac{1}{\sqrt{1+r}\pm \sqrt{1-r}} \text{ for all }i, j \right\}.
    \]
    By Proposition \ref{prop:hadamard-uni}, the fiber $\fK^{-1}(1)$ is precisely the set of all $2\times 2$ Hadamard unitary matrices. This suggests viewing the unitaries in $\fK^{-1}(r)$ as a natural generalization of Hadamard unitaries, which we refer to as \emph{$r$-Hadamard unitaries}. In particular, Theorem \ref{thm:dkk-r-char} yields the equivalent decomposition
    \[
    \sU(\mn{2}) = \bigsqcup_{0\le r\le 1}
    \left\{ u \in \sU(\mn{2}): \dkk (\Delta, u\Delta u^*) = r \right\}.
    \]
\end{remark}

The following result will be used in Corollary \ref{cor:d_KK_between_masa} to compute the Kadison--Kastler distance between $u\Delta u^*$ and $v\Delta v^*$ for two $2\times2$ unitary matrices $u,v$.

 \begin{theorem}[Unitary Invariance of Distance]\label{prop:dkk-gen-masa}
Let $u,v \in \sU(\mn{n})$ be two $n\times n$ unitaries. Then
\[
\dkk\left(u\Delta^{(n)} u^*, v\Delta^{(n)} v^*\right) = \dkk\left(\Delta^{(n)}, u^*v\Delta^{(n)}v^*u\right).
\]
\end{theorem}

\begin{proof}
It is enough to show that 
\[
\dkk \left(\Delta^{(n)}, u\Delta^{(n)} u^*\right) = \dkk\left(u^*\Delta^{(n)} u, \Delta^{(n)}\right).
\]
By definition,
\[
\dkk\left(\Delta^{(n)},u\Delta^{(n)}u^*\right) = \max \left\{\sup_{x \in \ball{\Delta^{(n)}}} d\left(x, \ball{u\Delta^{(n)} u^*}\right), \sup_{z\in \ball{\Delta^{(n)}}}d\left(uzu^*, \ball{\Delta^{(n)}}\right)\right\},
\]
where, $d\left(x, \ball{u\Delta^{(n)} u^*}\right)= \inf_{z \in \ball{\Delta^{(n)}}} \|x- uzu^*\|$.

But for any $x\in \ball{\Delta^{(n)}}$,
$$\|x- uzu^*\|= \|u^{*}xu- z\|.$$
Hence,
$$d\left(x, \ball{u\Delta^{(n)} u^*}\right)= \inf_{z \in \ball{\Delta^{(n)}}} \|x- uzu^*\|=  \inf_{z \in \ball{\Delta^{(n)}}} \|u^{*}xu- z\|= d\left(u^{*}xu, \ball{\Delta^{(n)}}\right).$$
This implies that
$$\sup_{x \in \ball{\Delta^{(n)}}}d\left(x, \ball{u\Delta^{(n)} u^*}\right)= \sup_{x \in \ball{\Delta^{(n)}}}d\left(u^{*}xu, \ball{\Delta^{(n)}}\right). $$
Similarly, for any $b\in \ball{\Delta^{(n)}}$, we have
\begin{align*}
d\left(ubu^*, \ball{\Delta^{(n)}}\right)&= \inf_{z\in \ball{\Delta^{(n)}}}\|ubu^*- z\|\\
&= \inf_{z\in \ball{\Delta^{(n)}}}\|b- u^{*}zu\|\\
&= d\left(b, \ball{u^{*}\Delta^{(n)}u}\right).
\end{align*}
This implies that
$$\sup_{b \in \ball{\Delta^{(n)}}}d\left(ubu^*, \ball{\Delta^{(n)}}\right)= \sup_{b \in \ball{\Delta^{(n)}}}d\left(b, \ball{u^{*}\Delta^{(n)} u}\right). $$
Hence,
$$\dkk\left(\Delta^{(n)}, u\Delta^{(n)} u^*\right) = \dkk\left(u^*\Delta^{(n)} u, \Delta^{(n)}\right).$$
Similarly, we can conclude 
$$\dkk\left(u\Delta^{(n)} u^*, v\Delta^{(n)} v^*\right) = \dkk\left(\Delta^{(n)}, u^*v \Delta^{(n)} v^*u\right).$$
\end{proof}

\begin{cor}\label{cor:d_KK_between_masa}
For any two $2\times 2$ unitary matrices $u,v\in \sU(\mn{2})$, we have
\[
\dkk(u\Delta u^*, v\Delta v^*)
=
2\left|\bar{u}_{11}v_{11}+\bar{u}_{21}v_{21}\right|
\left|\bar{u}_{11}v_{12}+\bar{u}_{21}v_{22}\right|.
\]
\end{cor}
\begin{proof}
     It follows from Propostion \ref{prop:dkk-gen-masa} and Theorem \ref{thm:dkk-2-case}.
\end{proof}

\begin{prop}\label{prop:dkk-proj}
Let $e$ be a rank-$k$ projection in $\mn{n}$ with $0<k<n$. Then
\[
\dkk\left(\Delta^{n},\,\C e\right)=1.
\]
\end{prop}
\begin{proof}
    It is enough to show that 
    $$\sup_{x \in \ball{\Delta^{n}}}\, \operatorname{dist}\left( x, \mathbb{T}e \right)\ge 1.$$
    Take $x = \mathbb{I}_n \in \mn{n}$. Then clearly $x \in \ball{\Delta^{n}}$. We claim that $\operatorname{dist}\left( x, \mathbb{T}e \right) =1$.

    By the spectral theorem, there exists some unitary matrix $u$ that diagonalizes $e$:
    $$e = u \operatorname{diag}(1, \dots, 1, 0, \dots, 0) u^*$$
    where there are $k$ many $1$'s and $n-k$ many $0$'s. Hence,
    $$x - \lambda e = u \operatorname{diag}(1-\lambda, \dots, 1-\lambda, 0, \dots, 0) u^*$$

    Hence, $\Vert{}x - \lambda e\Vert{} = \max(\vert{}1 - \lambda\vert{}, 1).$ Since $k>0$ one of the eigenvalues is strictly $1$. Thus 
    $$\operatorname{dist}\left( x, \mathbb{T}e \right) =\inf_{\vert{}\lambda\vert{} \le 1} \Vert{}x - \lambda e\Vert{} = 1.$$
    This completes the proof.
\end{proof}

\begin{theorem}\label{thm:dkk-partial-iso}
   Let $v$ be a partial isometry in $\mn{2}$. Then
\[
\dkk\left(\Delta,v\Delta v^*\right)=
\begin{cases}
2|v_{11}||v_{12}|, & \text{if } v^*v=1,\\
1, & \text{otherwise}.
\end{cases}
\]
\end{theorem}
\begin{proof}
    This follows from Proposition \ref{prop:dkk-proj} and Theorem \ref{thm:dkk-2-case}.
\end{proof}

\begin{cor}\label{cor:groupoid_d_KK}
Let $v$ be a partial isometry in $\mn{2}$. If
$v\in \grnor{\mn{2}}(\Delta)$, then
\[
\dkk\left(\Delta,v\Delta v^*\right)\in\{0,1\}.
\]
Moreover, for $u\in\grnor{\mn{2}}(\Delta)$, we have $\dkk\left(\Delta,u\Delta u^*\right)=0$
if and only if $u\in\unor{\mn{2}}(\Delta)$. Consequently, no Hadamard unitary belongs to $\grnor{\mn{2}}(\Delta)$.
\end{cor}
\begin{proof}
The first part follows from Theorem \ref{thm:dkk-partial-iso}. On the other hand for a Hadamard unitary the distance is exactly $1$ (see Theorem \ref{thm:hadamard-uni}), thus it cannot belong to the normaliser.    
\end{proof}

\begin{remark}
    Note that 
    \[
    v:=\frac{1}{2}\begin{pmatrix}
        1 &1\\1& 1
    \end{pmatrix}
    \]
    is a partial isometry such that $\dkk( \Delta, v \Delta v^*) = 1$ (from Theorem \ref{thm:dkk-partial-iso}) but $v \not \in \grnor{\mn{2}}(\Delta)$.
\end{remark}

\begin{prop}[{\cite[Theorem 4.4]{gupta2024}}]\label{prop:gupta-modified}
    Let $u=(\lambda_{ij}) \in \sU(\mn{2})$. Then
\[
\cos \alpha(\Delta, u\Delta u^*)
=
\sqrt{1-\left(2|\lambda_{11}||\lambda_{12}|\right)^2}.
\]
\end{prop}
\begin{proof}
    The proof of this fact is given in \cite[Theorem 4.4]{gupta2024}. However, there is a small typo in both the statement and the proof there. We correct the relevant part of the proof below. Since the argument is not new, we follow the same approach and notation as in \cite{gupta2024}, except for the calculation where the typo occurs. Let $T$ be the following $2\times 2$ matrix, as defined there:
    \[
    T:= \begin{pmatrix}
        \frac{|\lambda_{11}|^4 + |\lambda_{12}|^4 }{2} - \frac{1}{4} & \frac{\overline{\lambda_{21}} \lambda_{11} \left(|\lambda_{11}|^2 - |\lambda_{12}|^2 \right)}{2} \\
        \frac{\lambda_{21} \overline{\lambda_{11}} \left(|\lambda_{12}|^2 - |\lambda_{11}|^2 \right)}{2} & 
        \frac{|\lambda_{11}|^4 + |\lambda_{12}|^4 }{2} - \frac{1}{4}
    \end{pmatrix}
    \]
    To make the algebra cleaner, let $p = \vert{}\lambda_{11}\vert{}^2$. This means $\vert{}\lambda_{12}\vert{}^2 = 1 - p$ and $\vert{}\lambda_{21}\vert{}^2 = 1 - p$.
    
    Now set $a = (p - 1/2)^2$ and $z = T_{12}$. Then the matrix $T$ has the form
    $$T = \begin{pmatrix} a & z \\ -\bar{z} & a \end{pmatrix}$$

    A direct calculation gives
    \[
    T^*T = \left( a^2 + |z|^2\right) \mathbb{I}_2.
    \]
    Therefore:
    \[
    \|T\| = \sqrt{a^2 + |z|^2}= \sqrt{\frac{1}{4}\left(p - \frac{1}{2} \right)^2}= \frac{1}{4}|2p-1|= \frac{1}{4}\sqrt{\left( 2p-1 \right)^2}= \frac{1}{4}\sqrt{1-4p(1-p)}= \frac{1}{4}\sqrt{1-\left( 2 |\lambda_{11}| |\lambda_{12}| \right)^2}
    \]
    Following the remainder of the argument in \cite{gupta2024}, we obtain the desired result.
\end{proof}

The following result establishes the relationship between the Kadison--Kastler distance and the interior angle mentioned at the beginning of this section.

\begin{cor}\label{cor:relation_angle_distance}
    For every $2\times 2$ unitary matrix $u \in \mn{2}$, we have
\[
\dkk(\Delta, u\Delta u^*) = \sin \alpha(\Delta, u\Delta u^*).
\]
\end{cor}
\begin{proof}
    This follows from Theorem \ref{thm:dkk-2-case} and Proposition \ref{prop:gupta-modified}.
\end{proof}

\begin{cor}\label{cor:cont-dkk}
    If $(u_i)$ is a net of unitary matrices in $\sU(\mn{2})$ such that $u_i \to 1$, then
\[
\dkk(\Delta, u_i\Delta u_i^*) \to 0,
\]
and consequently,
\[
\alpha(\Delta, u_i\Delta u_i^*) \to 0.
\]
\end{cor}
\begin{proof}
    This directly follows from the Corollary \ref{cor:relation_angle_distance}.
\end{proof}

    \section{Relationship Between $\alpha$, $h$, $\lambda$, and $\dkk$ in $\mn{2}$}
    \label{sec:relationship}
    
The main purpose of this section is to establish the relationship between three well-known quantities---the interior angle $\alpha$, the Connes--St\o rmer modified entropy $h(-\mid-)$, and the Pimsner--Popa probabilistic index $\lambda$---and the Kadison--Kastler distance $\dkk$. In Corollary \ref{cor:equival}, we show that for two masas in $\mn{2}$, the Kadison--Kastler distance is maximal if and only if their interior angle is $\frac{\pi}{2}$, if and only if their modified relative entropy is maximal, if and only if their Pimsner--Popa constant is $\frac{1}{2}$. Moreover, Theorem \ref{thm:equivalence} shows that these equivalent conditions hold if and only if $u^*v$ is a Hadamard unitary.

We begin by introducing some basic notation and conventions that will be used throughout the remainder of the paper.

The conditional expectation $E_0:\mn{n}\to\C\otimes\mathbb{I}_n\cong\C$ is given by
\[
E_0(X):=\frac{1}{n}\operatorname{Tr}(X)\otimes\mathbb{I}_n.
\]
The conditional expectation $E_{\Delta^{(n)}}:\mn{n}\to\Delta^{(n)}$ is given by
\[
E_{\Delta^{(n)}}(X)
=
\operatorname{diag}(X_{11},\dots,X_{nn}).
\]
For any unitary $u\in\mn{n}$, we define the conditional expectation
$E_u:\mn{n}\to u\Delta^{(n)}u^*$ by
\begin{equation}\label{eq:unitary-conditional}
    E_u:=\operatorname{Ad}_u\circ E_{\Delta^{(n)}}\circ\operatorname{Ad}_{u^*}.
\end{equation}
Thus,
\[
E_u(X)=uE_{\Delta^{(n)}}(u^*Xu)u^*
\]
for all $X\in\mn{n}$.

\begin{lem}\label{lem:surj}
For any unitary $v \in \mn{n}$, the map
\[
X \mapsto E_{\Delta^{(n)}}(v^*Xv)
\]
is surjective onto $\Delta^{(n)}$.
\end{lem}

\begin{lem}\label{lem:exp-for}
    Let $W$ be an $n\times n$ unitary matrix. For any $D=\operatorname{diag}(d_1,\dots,d_n)\in\Delta^{(n)}$, the conditional expectation onto $\Delta^{(n)}$ is given by
\[
E_{\Delta^{(n)}}(WDW^*)
=
\operatorname{diag}\left(
d_1|W_{11}|^2+\cdots+d_n|W_{1n}|^2,\,
\dots,\,
d_1|W_{n1}|^2+\cdots+d_n|W_{nn}|^2
\right).
\]
\end{lem}

We pause briefly to recall the notion of a \emph{commuting square} of subalgebras of a von Neumann algebra. Let $\sN\subseteq\sM$ be an inclusion of finite von Neumann algebras, and let $E_\sN:\sM\to\sN$ be a conditional expectation. Let $\sP,\sQ$ be intermediate subalgebras with corresponding conditional expectations $E_\sP:\sM\to\sP$ and $E_\sQ:\sM\to\sQ$, respectively. We say that the quadruple $(\sM,\sP,\sQ,\sN)$ forms a \emph{commuting square} if
\[
E_\sP\circ E_\sQ=E_\sN=E_\sQ\circ E_\sP.
\]

We will also use the following definition.

\begin{definition}[Orthogonal Masas, {\cite{Choda2008}}]
    Two masas $\sC,\sD$ of the von Neumann algebra $\mn{n}$ are said to be \emph{orthogonal in the sense of Popa} if $(\mn{n},\sC,\sD,\C)$ is a commuting square and
    \[
    \sC\cap\sD=\C.
    \]
\end{definition}

With this terminology in place, we prove the following result, which gives a precise characterization of the relationship between commuting squares and Hadamard unitaries.

\begin{prop}\label{prop:exp-and-hada-uni}
    Let $u, v \in \sU(\mn{n})$ be unitaries. Then $E_u \circ E_v = E_0$ if and only if $u^*v$ is a Hadamard unitary.
\end{prop}
\begin{proof}
    Let $W = u^* v$. By definition, evaluating the composition on an arbitrary matrix $X$ yields:$$(E_u \circ E_v)(X) = u E_{\Delta^{(n)}}(W E_{\Delta^{(n)}}(v^* X v) W^*) u^*$$
    Let $D = E_{\Delta^{(n)}}(v^* X v)$, which is diagonal. Setting $(E_u \circ E_v)(X) = E_0(X)$ results in the equation:
    $$u E_{\Delta^{(n)}}(W D W^*) u^* = \frac{\text{Tr}(X)}{n} \otimes\mathbb{I}_n$$
    Conjugating both sides by $u^*$ leaves it unchanged:
    $$E_{\Delta^{(n)}}(W D W^*) = \frac{\text{Tr}(X)}{n}\otimes \mathbb{I}_n$$
    
    Note that $\text{Tr}(X) = \text{Tr}(v^* X v) = \text{Tr}(E_{\Delta^{(n)}}(v^* X v)) = \text{Tr}(D)$. Letting $D = \text{diag}(d_1,\dots, d_n)$, we have $\text{Tr}(X) = d_1 +\cdots + d_n$. Substituting this alongside Lemma \ref{lem:exp-for} transforms our equality into a purely diagonal system:
    
    $$\operatorname{diag}\left(d_1|W_{11}|^2 + \dots + d_n |W_{1n}|^2, \dots,  d_1|W_{n1}|^2 + \dots + d_n |W_{nn}|^2\right) = \frac{d_1 +\cdots+ d_n}{n} \otimes \mathbb{I}_n 
    $$

    ($\Rightarrow$):
If the composition equals $E_0$, this matrix equation must hold for all $X \in \mn{n}$, and thus (by Lemma \ref{lem:surj}) for any values of $d_1, \dots, d_n$.

For all $j$, setting $d_j = 1$ and $d_i= 0$ for $i \ne j$, forces $\vert{}W_{ij}\vert{}^2 = 1/n$ for all $i$. Thus $W$ is a Hadamard unitary by Proposition \ref{prop:hadamard-uni}.

($\Leftarrow$):
If $W = u^*v$ is a Hadamard unitary, then $\vert{}W_{ij}\vert{}^2 = 1/n$ for all $i,j$. Substituting this into the left side of our diagonal system gives:
$$ \frac{d_1 +\cdots +d_n}{n} \otimes \mathbb{I}_n= \frac{\text{Tr}(X)}{n} \otimes \mathbb{I}_n$$

Tracing our algebraic steps in reverse confirms that $(E_u \circ E_v)(X) = E_0(X)$ for all $X$.
\end{proof}

\begin{prop}\label{prop:intersection_trivial}
   Let $u,v\in \sU(\mn{n})$ be unitaries such that $u^*v$ is a Hadamard unitary. Then
\[
u\Delta^{(n)}u^*\cap v\Delta^{(n)}v^*=\mathbb{C}\mathbb{I}_n.
\]
\end{prop}
\begin{proof}
    Let $X \in u\Delta^{(n)} u^* \cap v\Delta^{(n)} v^*$. Therefore, there are diagonal matrices $D_1, D_2 \in \Delta^{(n)}$ such that$$X = u D_1 u^* = v D_2 v^*$$
    
    By multiplying on the left by $v^*$ and on the right by $v$, we can isolate $D_2$:
    $$D_2 = (v^* u) D_1 (u^* v)$$

    Let $W = u^* v$. By our hypothesis, $W$ is a Hadamard unitary matrix, which means $\vert W_{ij}\vert = 1/\sqrt{n}$ for all $i,j$. Since $W$ is unitary, its adjoint is $W^* = v^* u$. Therefore, we have:$$D_2 = W^* D_1 W$$

    Let the entries of the diagonal matrices be $D_2 = \operatorname{diag}(\lambda_1, \dots, \lambda_n)$ and $D_1 = \operatorname{diag}(\mu_1, \dots, \mu_n)$. Let us compute the $k$-th diagonal entry of $D_2$:
    $$(D_2)_{kk} = (W^* D_1 W)_{kk} = \sum_{i=1}^n \sum_{j=1}^n (W^*)_{ki} (D_1)_{ij} W_{jk}$$

    Because $D_1$ is a diagonal matrix, its off-diagonal entries are zero (meaning $(D_1)_{ij} = 0$ for $i \neq j$). This collapses the double sum into a single sum:$$\lambda_k = \sum_{i=1}^n (W^*)_{ki} (D_1)_{ii} W_{ik}$$

    By the definition of the adjoint matrix, $(W^*)_{ki} = \overline{W_{ik}}$. Substituting this and $(D_1)_{ii} = \mu_i$, we get:
    $$\lambda_k = \sum_{i=1}^n \overline{W_{ik}} \mu_i W_{ik} = \sum_{i=1}^n \vert W_{ik}\vert^2 \mu_i$$

    Therefore, $$\lambda_k = \sum_{i=1}^n \frac{1}{n} \mu_i = \frac{1}{n} \operatorname{Tr}(D_1)$$

    This means every single diagonal entry of $D_2$ is identical. Let this common value be $\lambda$. This forces the diagonal matrix $D_2$ to be a scalar multiple of the identity:
    $$D_2 = \begin{pmatrix} \lambda & & 0 \\ & \ddots & \\ 0 & & \lambda \end{pmatrix} = \lambda \mathbb{I}_n$$
    
    Finally, substituting this back into our original expression for $X$:
    $$X = v D_2 v^* = v (\lambda \mathbb{I}_n) v^* = \lambda (v v^*) = \lambda \mathbb{I}_n$$
    
    Therefore, $X \in \mathbb{C}\mathbb{I}_n$, which completes the proof.
\end{proof}

\begin{remark}
The converse of the last result does not hold in general. Let $u=\mathbb{I}_2$ and let $v$ be the standard rotation matrix through the angle $\pi/3$, given by
\[
v=
\begin{pmatrix}
\frac{1}{2} & -\frac{\sqrt{3}}{2}\\
\frac{\sqrt{3}}{2} & \frac{1}{2}
\end{pmatrix}.
\]
Then $u^*v=v$ is not a Hadamard unitary. Nevertheless, $\Delta\cap v\Delta v^*=\mathbb{C}\mathbb{I}_2.$
\end{remark}

We next establish the unitary invariance of the Pimsner--Popa probabilistic constant, which will be used in the subsequent theorem.

\begin{lem}\label{lem:popa-index-invariance}
For any two $n\times n$ unitary matrices $u,v\in\mn{n}$, we have
\[
\lambda(u\Delta^{(n)} u^*,v\Delta^{(n)} v^*)
=
\lambda(\Delta^{(n)},u^*v\Delta^{(n)} v^*u).
\]
\end{lem}
\begin{proof}
Let $E_u, E_v$ and $E_{u^*v}$ denote the conditional expectations from $\mn{n}$ onto $u\Delta^{(n)} u^*, v\Delta^{(n)} v^*$ and $u^*v \Delta^{(n)} v^*u$ respectively, for the unitaries $u,v$ and $u^*v$, as defined in equation \ref{eq:unitary-conditional}. Then 
\begin{align*}
   \lambda(u\Delta^{(n)} u^*, v\Delta^{(n)} v^*)
   &= \sup \left\{ t \geq 0 \,\middle|\,\,\, E_v(x) \geq tx
   \quad \forall\, x \in (u\Delta^{(n)} u^*)_{+} \right\} \\
   &= \sup \left\{ t \geq 0 \,\middle|\,\,\, vE_{\Delta^{(n)}}(v^*xv)v^* \geq tx
   \quad \forall\, x \in (u\Delta^{(n)}_{+} u^* \right\}\\
    &= \sup \left\{ t \geq 0 \,\middle|\,\,\, E_{\Delta^{(n)}}(v^*xv) \geq tv^*xv
   \quad \forall\, x \in (u\Delta^{(n)}_{+} u^* \right\}\\
    &= \sup \left\{ t \geq 0 \,\middle|\,\,\, E_{\Delta^{(n)}}(v^*uyu^*v) \geq tv^*uyu^*v
   \quad \forall\, y \in \Delta^{(n)}_{+} \right\}\\
   &= \sup \left\{ t \geq 0 \,\middle|\,\,\, u^*vE_{\Delta^{(n)}}(v^*uyu^*v)v^*u \geq ty
   \quad \forall\, y \in \Delta^{(n)}_{+} \right\}\\
   &= \sup \left\{ t \geq 0 \,\middle|\,\,\, E_{u^*v}(y) \geq ty
   \quad \forall\, y \in \Delta^{(n)}_{+} \right\}\\
   &=\lambda(\Delta^{(n)}, u^*v \Delta^{(n)} v^*u).
   \end{align*}
\end{proof}

\begin{definition}[Hamming Number]\label{def:hamming-num}
    Given a nonzero vector $\Vec{u} \in \C^n$, the Hamming number is given by 
    \[
    \|\Vec{u}\|_0:= \text{ number of nonzero entries in } \Vec{u}.
    \]
\end{definition}

\begin{theorem}\label{thm:equivalence}
   Let $u,v\in\sU(\mn{2})$ be two $2\times 2$ unitary matrices. Then the following are equivalent:
\begin{itemize}
    \item[(i)] $\dkk\left(u\Delta u^*,v\Delta v^*\right)=1$.
    
    \item[(ii)] $u^*v$ is a Hadamard unitary.
    
    \item[(iii)] $\alpha\left(u\Delta u^*,v\Delta v^*\right)=\frac{\pi}{2}$.
    
    \item[(iv)] The following diagram is a commuting square:
    \[
    \begin{array}{ccc}
    u\Delta u^* & \subset & \mathbb{M}_2(\mathbb{C}) \\
    \cup && \cup \\
    \mathbb{C} & \subset & v\Delta v^*
    \end{array}
    \]
    
    \item[(v)] $h(u\Delta u^* \mid v\Delta v^*)=\log(2)$.
    
    \item[(vi)] $\lambda(u\Delta u^*,v\Delta v^*)=\frac{1}{2}$.
\end{itemize}
\end{theorem}

\begin{proof}
 This is known that $(iii) \iff (iv)$ (see, for example \cite[Remark 5.4]{Bakshi_gupta_2021}).

 For $(i) \iff (ii)$ note the following equation together with Proposition \ref{prop:dkk-gen-masa}:
 \[
 \dkk\left( u\Delta u^*, v \Delta v^*\right) = \dkk\left( \Delta, u^*v \Delta v^*u\right). 
 \]
 $(ii)\iff (iv)$ follows from Proposition \ref{prop:exp-and-hada-uni}.

 $(ii) \implies (v)$. Assume $(ii)$, then we have proved that $(iv)$ holds. On the other hand $(ii)$ implies (by Proposition \ref{prop:intersection_trivial}) that $u\Delta u^* \cap v\Delta v^* = \C.$ Now by \cite[Corollary 3.3]{Choda2008}, we get (v).

 $(v)\implies (ii)$. By \cite[Corollary 3.3]{Choda2008}, $(v)$ implies $(iv)$ and which we have already proved implies $(ii)$.

 \noindent($(i) \implies (vi)$). Let $\dkk(u\Delta u^*, v\Delta v^*) = 1$. Since we've already shown that this implies $(ii)$, that is, $v^*u$ is a Hadamard unitary. By Lemma \ref{lem:popa-index-invariance} and \cite[Theorem 4.2]{BakshiGuin2025}, we get
 \[
 \lambda(u\Delta u^*, v \Delta v^*) = \lambda(\Delta, u^*v \Delta v^*u) = \min_{1\le i \le 2} \left( \|(v^*u)_i\|_0 \right)^{-1}= \frac{1}{2}.
 \]

 \noindent($(vi) \implies (i)$). Suppose $\lambda(u\Delta u^*, v\Delta v^*) = \frac{1}{2}$. Then by Lemma \ref{lem:popa-index-invariance} and \cite[Theorem 4.2]{BakshiGuin2025}, we have
 \[
 \min_{1\le i \le 2} \left( \|(v^*u)_i\|_0 \right)^{-1}=\lambda(\Delta, u^*v \Delta v^*u)=\lambda(u\Delta u^*, v\Delta v^*)=\frac{1}{2}.
 \]
 Note that since $v^*u$ is a $2 \times 2$ unitary matrix so the Hamming number of each column is equal and non-zero, that is, $\|(v^*u)_1\|_0 = \|(v^*u)_2\|_0=: g \ne 0$. Therefore $g = 2$ which means all the entries of $v^*u$ and hence those of $u^*v$ are non-zero. Then by \cite[Corollary 4.3]{BakshiGuin2025}, the pair $(\Delta, u^*v\Delta v^*u)$ is orthogonal in sense of Popa, that is, we have the following commuting square:
\begin{align*}
\begin{array}{ccc}
\Delta & \subset & \mathbb{M}_2(\mathbb{C}) \\
\cup && \cup \\
\mathbb{C} & \subset & u^*v\Delta v^*u
\end{array}
\end{align*}
 But we have already shown that $(iv)$ implies $(ii)$. So applying `$(iv) \implies (ii)$' on this commuting square we get $(ii)$ which in turn implies $(i)$ as we have proved already. 
 
 This completes the proof!
\end{proof}

\begin{cor}\label{cor:equival}
Let $u \in \mn{2}$ be an arbitrary unitary matrix. Then the following are equivalent:
\begin{itemize}
    \item[(i)] $\dkk\left(\Delta,u\Delta u^*\right)=1$.
    \item[(ii)] $\lambda\left(\Delta,u\Delta u^*\right)=\frac{1}{2}$.
    \item[(iii)] $h\left(\Delta\mid u\Delta u^*\right)=\log(2)$.
    \item[(iv)] $\alpha\left(\Delta,u\Delta u^*\right)=\frac{\pi}{2}$.
\end{itemize}
\end{cor}

The following result is known and follows from Choda's work but we're keeping it here for the sake of completeness.
\begin{theorem}
    Let $u,v$ be two $n\times n$ unitaries such that $u^*v$ is Hadamard. Then
    \[
        H(u\Delta^{(n)}u^* \mid v\Delta^{(n)}v^*)=\log(n).
    \]
\end{theorem}

\begin{proof}
    Since $u^*v$ is Hadamard, Propositions \ref{prop:exp-and-hada-uni} and
    \ref{prop:intersection_trivial} imply that
    $(u\Delta^{(n)}u^*,v\Delta^{(n)}v^*)$ is an orthogonal pair in the sense of Popa.
    Hence, by \cite[Corollary 3.3]{Choda2008}, the modified Connes--St\o rmer
    entropy satisfies
    \[
        h(u\Delta^{(n)}u^* \mid v\Delta^{(n)}v^*)=\log(n).
    \]
    Finally, by \cite[Corollary 2.3]{Choda2011}, we have
    \[
        H(u\Delta^{(n)}u^* \mid v\Delta^{(n)}v^*)
        =
        h(u\Delta^{(n)}u^* \mid v\Delta^{(n)}v^*),
    \]
    and the result follows.
\end{proof}

    \section{Kadison Kastler Distance in $\mn{n}$ for $n \ge 3$}
    \label{sec:dis-in-ndim}

\begin{lem}\label{lem:dir-dis-masa -rot}
    Let $\Deltan $ be the diagonal masa  in $\mn{n}$. Let $R_{jk}(\theta)$ be an $n \times n$ rotation matrix in the $j$-$k$ principal plane (where $1 \le j < k \le n$) by an angle $\theta$. For a fixed diagonal matrix $D_1 = \operatorname{diag}(x_1, \dots, x_n)$ in the closed unit ball of $\Deltan $, the infimum distance to the unit ball of the rotated algebra evaluates exactly to $\frac{1}{2}\vert{}x_j - x_k\vert{}\vert{}\sin(2\theta)\vert{}$.
\end{lem}
\begin{proof}
    By exploiting the unitary invariance of the operator norm, we can multiply the expression inside the norm by the adjoint unitary on the left to rewrite the infimum as $\inf_{D_2 \in (\Deltan )_1} \Vert{} R_{jk}(\theta)^* D_1 R_{jk}(\theta) - D_2 \Vert{}_{op}$.
    
    The conjugation by $R_{jk}(\theta)$ alters only the $2 \times 2$ submatrix corresponding to the indices $j$ and $k$, leaving the other $n-2$ diagonal entries of $D_1$ completely isolated and unchanged.
    
    Because the operator norm of a block-diagonal matrix is simply the maximum of the operator norms of its individual blocks, we can minimize the distance by choosing a $D_2$ that perfectly matches the $n-2$ unchanged diagonal entries, driving their local differences to exactly 0.

    The $2 \times 2$ submatrix evaluates to:$$\begin{pmatrix} x_j \cos^2\theta + x_k \sin^2\theta & (x_j - x_k)\cos\theta\sin\theta \\ (x_j - x_k)\cos\theta\sin\theta & x_j \sin^2\theta + x_k \cos^2\theta \end{pmatrix}$$
    
    For the $2 \times 2$ rotated block, the diagonal entries evaluate to $x_j \cos^2\theta + x_k \sin^2\theta$ and $x_j \sin^2\theta + x_k \cos^2\theta$.
    
    Because $D_1$ is in the unit ball, $\vert{}x_j\vert{} \le 1$ and $\vert{}x_k\vert{} \le 1$, meaning these convex combinations are also bounded by 1; therefore, choosing the remaining two entries of $D_2$ to perfectly match this main diagonal is strictly valid within the unit ball constraint.
    
    Subtracting this optimal $D_2$ eliminates the main diagonal entirely, leaving a $2 \times 2$ block with off-diagonal entries exactly equal to $(x_j - x_k)\cos\theta\sin\theta$.
    
    Using the double-angle identity, the operator norm of this isolated off-diagonal matrix is exactly $\frac{1}{2}\vert{}x_j - x_k\vert{}\vert{}\sin(2\theta)\vert{}$.
\end{proof}

We must now maximize this infimum distance over all elements in the closed unit ball of $\Deltan$.

\begin{lem}\label{lem:sin2theta}
    The supremum of the quantity $\frac{1}{2}\vert{}x_j - x_k\vert{}\vert{}\sin(2\theta)\vert{}$ over all elements $D_1 \in (\Deltan )_1$ evaluates to exactly $\vert{}\sin(2\theta)\vert{}$.
\end{lem}
\begin{proof}
    Because $D_1$ is restricted to the closed unit ball, its individual diagonal entries are constrained to the closed unit disk ($\vert{}x_j\vert{} \le 1$ and $\vert{}x_k\vert{} \le 1$).

    Using the triangle inequality, the maximum possible value for the distance $\vert{}x_j - x_k\vert{}$ is bounded above by $\vert{}x_j\vert{} + \vert{}-x_k\vert{} \le 2$.

    Substituting this constraint yields an upper bound of $\frac{2}{2}\vert{}\sin(2\theta)\vert{} = \vert{}\sin(2\theta)\vert{}$.

    This maximum bound is physically achievable by selecting two points on the unit disk that are on exactly opposite sides of the origin, such as $x_j = 1$ and $x_k = -1$.

    Because the upper bound is achievable within the domain, the supremum over the entire unit ball is exactly $\vert{}\sin(2\theta)\vert{}$.
\end{proof}

With both the infimum and supremum established, we can now formally state the exact Kadison--Kastler distance between masa s conjugated by principal rotations.

\begin{theorem}\label{thm:dkk-for-rotations}
    Let $R_{jk}(\theta)$ be a rotation matrix in any $j$-$k$ principal plane (for $1 \le j < k \le n$). The Kadison-Kastler distance $\dkk(\Deltan , R_{jk}(\theta)\Deltan R_{jk}(\theta)^*) = \vert{}\sin(2\theta)\vert{}$.
\end{theorem}
\begin{proof}
    The Kadison-Kastler distance is defined as the maximum of the two directed Hausdorff distances between the closed unit balls of the respective $C^*$-subalgebras.

    By combining Lemma \ref{lem:dir-dis-masa -rot} and Lemma \ref{lem:sin2theta}, the directed distance from the standard diagonal masa  $\Deltan $ to the rotated algebra evaluates to exactly $\vert{}\sin(2\theta)\vert{}$.

    For the reverse directed distance, the infimum approximation relies on the adjoint rotation matrix $R_{jk}(\theta)^*$, which implements a rotation by $-\theta$.

    Because the absolute value of the sine function is symmetric, the reverse directed distance evaluates to an identical supremum of $\vert{}\sin(-2\theta)\vert{} = \vert{}\sin(2\theta)\vert{}$.

    Evaluating the maximum of these two identical directed distances yields a final Kadison-Kastler distance of exactly $\vert{}\sin(2\theta)\vert{}$.
\end{proof}

As a consequence of this exact formula, we can immediately demonstrate that the distance between conjugate masas is not quantized, but rather smoothly covers the entire interval $[0, 1]$.

\begin{theorem}\label{thm:dkk-attains-r}
    For any $r \in [0, 1]$ there exists a $n\times n$ unitary matrix $u\in \mn{n}$ such that
    \[
    \dkk (\Deltan , u\Deltan  u^*) = r.
    \]
\end{theorem}
\begin{proof}
    Choose $u$ to be the rotation matrix in any $jk$ plane by the angle 
    $$\theta = \frac{1}{2} \arcsin(r).$$
\end{proof}

\begin{remark}\label{rem:dkk-max-by-non-hada}
    There are unitaries which are not Hadamard in $\mn{n}$ for $n>2$ such that 
    \[
    \dkk (\Deltan , u\Deltan  u^*) = 1.
    \]
    For example consider 
    $$R_{12}(\pi/4) = \begin{pmatrix} \cos(\pi/4) & -\sin(\pi/4) & 0 \\ \sin(\pi/4) & \cos(\pi/4) & 0 \\ 0 & 0 & 1 \end{pmatrix} = \begin{pmatrix} \frac{1}{\sqrt{2}} & -\frac{1}{\sqrt{2}} & 0 \\ \frac{1}{\sqrt{2}} & \frac{1}{\sqrt{2}} & 0 \\ 0 & 0 & 1 \end{pmatrix}$$
    This is not at all true for $n=2$ case!
\end{remark}

We now shift our attention to Hadamard unitaries, which represent maximum quantum complementarity. To establish bounds on the Kadison--Kastler distance for these highly complementary subalgebras, we first recall a standard norm inequality bounding the operator norm by the Frobenius norm.

\begin{lem}\label{lem:op-norm-frob-norm}
    For any matrix $X \in \mn{n}$, $\Vert{}X\Vert{}_F^2 \le n \Vert{}X\Vert{}_{op}^2$. 
\end{lem}
\begin{proof}
    The Frobenius norm squared, $\Vert{}X\Vert{}_F^2$, is exactly the sum of the squares of the $n$ singular values of $X$. The operator norm squared, $\Vert{}X\Vert{}_{op}^2$, is the maximum squared singular value. Because the sum of $n$ non-negative values is strictly bounded above by $n$ times their maximum value, the inequality holds.
\end{proof}

We deploy this inequality to bound the directed distance from the standard masa  to the one generated by a Hadamard unitary.

\begin{lem}\label{lem:dir-kk-dis-hadamard}
    For any Hadamard unitary $u \in \mn{n}$ where $n \ge 2$, the directed Hausdorff distance from $\Deltan $ to $u\Deltan u^*$ is at least $1$.
\end{lem}
\begin{proof}
    Let $A = \Deltan $. By exploiting unitary invariance, the directed Hausdorff distance can be written as $\sup_{D_1 \in (A)_1} \inf_{D_2 \in (A)_1} \Vert{}u^* D_1 u - D_2\Vert{}_{op}$.

Let $\omega_k = e^{2\pi i (k-1)/n}$ be the $n$-th roots of unity. We select the specific test element $D_1 = \operatorname{diag}(\omega_1, \omega_2, \dots, \omega_n)$. Because $\vert{}\omega_k\vert{} = 1$ for all $k$, this matrix resides perfectly within the closed unit ball $(A)_1$.
Let $M = u^* D_1 u$. By the definition of an $n \times n$ Hadamard matrix, every entry $u_{j,k}$ has a uniform magnitude of $1/\sqrt{n}$, meaning $\vert{}u_{j,k}\vert{}^2 = 1/n$. We compute the $j$-th diagonal entry of $M$:
$$M_{j,j} = \sum_{k=1}^n (u^*)_{j,k} (D_1)_{k,k} u_{k,j} = \sum_{k=1}^n \vert{}u_{k,j}\vert{}^2 \omega_k = \frac{1}{n} \sum_{k=1}^n \omega_k$$

Because the sum of all $n$-th roots of unity evaluates to $0$ for $n \ge 2$, the matrix $M$ has a strictly zero main diagonal.

The Frobenius norm is unitarily invariant, meaning $\Vert{}M\Vert{}_F^2 = \Vert{}D_1\Vert{}_F^2 = \sum_{k=1}^n \vert{}\omega_k\vert{}^2 = n$. Because $M$ has a zero main diagonal, the sum of the squared magnitudes of its off-diagonal entries is exactly $n$.We now subtract an arbitrary diagonal matrix $D_2 = \operatorname{diag}(y_1, y_2, \dots, y_n) \in (A)_1$. This operation subtracts $y_j$ from the zero main diagonal of $M$ but leaves the off-diagonal entries completely untouched. The Frobenius norm squared of the resulting difference evaluates to:$$\Vert{}M - D_2\Vert{}_F^2 = \sum_{j \neq k} \vert{}M_{j,k}\vert{}^2 + \sum_{j=1}^n \vert{}-y_j\vert{}^2 = n + \sum_{j=1}^n \vert{}y_j\vert{}^2$$

Because squared magnitudes are always non-negative, $\Vert{}M - D_2\Vert{}_F^2 \ge n$ for any valid choice of $D_2$.

Applying Lemma \ref{lem:op-norm-frob-norm} to this difference yields:
$$n \Vert{}M - D_2\Vert{}_{op}^2 \ge \Vert{}M - D_2\Vert{}_F^2 \ge n$$

Dividing by $n$ provides $\Vert{}M - D_2\Vert{}_{op}^2 \ge 1$, which algebraically guarantees that $\Vert{}M - D_2\Vert{}_{op} \ge 1$. Since the infimum distance for this single test element is $\ge 1$, the supremum over the entire unit ball is bounded below by $1$.
\end{proof}

Note that for $n\ge 3$, there always exist non-Hadamard unitaries $u$ in $\mn{n}$ such that $\dkk (\Deltan, u\Deltan u^*) = 1$. Indeed, consider the rotation matrix $u = R_{jk}(\pi/4)$ for any choice of $j\neq k$. Since $n\ge 3$, $u$ is not Hadamard, while Theorem \ref{thm:dkk-for-rotations} gives
\[
\dkk (\Deltan, u\Deltan u^*) = 1.
\]
Thus, the implication `(i) $\implies$ (ii)' in Theorem \ref{thm:equivalence} does not hold for $n\ge 3$. However, we prove below that the converse implication is always true in $\mn{n}$:

\begin{theorem}\label{thm:hadamard-makes-d-1}
    For any Hadamard unitary $u \in \mn{n}$ where $n \ge 2$, 
    $$\dkk(\Deltan , u\Deltan u^*) = 1.$$
\end{theorem}
\begin{proof}
    By Lemma \ref{lem:dir-kk-dis-hadamard}, the directed Hausdorff distance from $\Deltan $ to $u\Deltan u^*$ is bounded below by $1$. Because the Kadison-Kastler distance takes the maximum of the two directed distances, the overall distance $\dkk \ge 1$. Thus the distance is exactly $1$.
\end{proof}

Building on the fact that Hadamard unitaries yield a maximal distance of $1$, we can establish our surjectivity result (Theorem \ref{thm:dkk-attains-r}) more broadly using topological arguments. The intermediate value theorem allows us to show that pairs of conjugate masas can realize any distance $r \in [0, 1]$.

\begin{theorem}\label{thm:dkk-not-quantized}
    For every $r\in[0,1]$, there exist $n\times n$ unitaries $u,v\in\mn{n}$ such that
\[
\dkk\left(u\Delta^{(n)}u^*,v\Delta^{(n)}v^*\right)=r.
\]
\end{theorem}
\begin{proof}
Indeed, we shall prove that there always exists a unitary matrix $v\in \mn{n}$ such that
\[
    \dkk(\Delta^{(n)}, v\Delta^{(n)}v^*)=r.
\]

To this end, define the function
\begin{align*}
    f : \sU(\mn{n}) &\longrightarrow [0, \infty)\\
    v &\longmapsto \dkk(\Deltan, v\Deltan v^*).
\end{align*}

Fix a Hadamard unitary $v_0\in \mn{n}$; for example, let
\[
    (v_0)_{jk}=\frac{1}{\sqrt{n}}e^{\frac{2\pi ijk}{n}}
\]
for all $j,k$. By Theorem \ref{thm:hadamard-makes-d-1}, we have $f(v_0)=1$. On the other hand, $f(\mathbb{I}_n)=0$.

Since $\sU(\mn{n})$ is path connected and $f$ is clearly continuous, the intermediate value theorem implies that, for every $r\in[0,1]$, there exists a unitary matrix $v\in \sU(\mn{n})$ such that $f(v)=r$. This completes the proof.
\end{proof}

    \section{Interior Angle Between MASAs in $\mn{n}$}
    \label{sec:mn}
    The main goal of this section is to derive an explicit formula for the interior angle between two masas in the matrix algebra $\mn{n}$ for arbitrary $n\geq 2$. Since the computation of the interior angle involves the basic constructions associated with the masas and their corresponding Jones projections, we begin by describing these objects explicitly.

In particular, we are interested in computing
\[
\alpha\left(u\Delta^{(n)}u^*,v\Delta^{(n)}v^*\right),
\]
where $u,v\in\mn{n}$ are unitary matrices. To this end, we need an explicit description of the Jones projection, denoted by $e^u$, associated with the inclusion $u\Delta^{(n)}u^*\subseteq\mn{n}$. This was established in Proposition \ref{prop:jones-proj}. Moreover, for masas of this form, their corresponding Jones projections are unitarily equivalent in the algebra $\mn{n}\otimes\mn{n}$; see Corollary \ref{cor:equval-of-jones-proj}.

In Theorem \ref{thm:uni-inv-angle}, we established that the interior angle satisfies the same unitary invariance property as the Kadison--Kastler distance $\dkk$ (see Theorem \ref{prop:dkk-gen-masa}). As an application, Theorem \ref{thm:dkk-is-sin} establishes, in the $2\times2$ case, the following relationship between the Kadison--Kastler distance and the interior angle (compare Corollary \ref{cor:relation_angle_distance}):
\[
\mathrm{d}_{\textrm{KK}}(u\Delta u^*,v\Delta v^*)
=\sin\alpha(u\Delta u^*,v\Delta v^*).
\]
Furthermore, in Theorem \ref{thm:angle-hadamard-comm-sq}, we showed that the interior angle is maximal if and only if $u^*v$ is a Hadamard unitary. Combining these results, we obtain in Theorem \ref{thm:angle-betwn-masas} an explicit formula for the interior angle between two masas in $\mn{n}$.

Finally, this formula allows us to show that, for arbitrary $n\geq2$, the interior angle between two masas can attain every value in $[0,\pi/2]$. This generalizes the result of \cite{gupta2024}, where the corresponding statement was established for the case $n=2$.

Let $\mn{n}$ be the algebra of $n \times n$ complex matrices. Equip $\mn{n}$ with the inner product $\langle X, Y \rangle_\tau := \tau(X^*Y)$, where $\tau(Z) = \frac{1}{n} \operatorname{Tr}(Z)$, making it a finite-dimensional Hilbert space.

\begin{theorem}\label{thm:basic-c-ten-I}
    Consider the inclusion of $C^*$-algebras $\C \otimes \mathbb{I}_n\subseteq \mn{n}$, equipped with the trace-preserving conditional expectation
    \[
    E_0:= \frac{1}{n}\operatorname{Tr}\otimes \mathbb{I}_n.
    \]
    Then the following statements hold:
    \begin{itemize}
    \item[(i)] $E_0$ is of finite Watatani index with quasi-basis $\left\{u_{ij} := \sqrt{n}E_{ij}: i,j = 1, \dots, n\right\}$. Moreover,
    $$\operatorname{Ind}(E_0)^{-1} = \frac{1}{n^2} \mathbb{I}_n.$$
        \item[(ii)] The map $\Phi: \mn{n}\otimes \mn{n}\to \sB\left( \mn{n}, \langle -, - \rangle_\tau\right)$, defined by
\[
\Phi(A\otimes B):= A(-)B^T,
\]
is a $C^*$-isomorphism. Moreover, the reduced $C^*$-basic construction associated with the inclusion $\C\otimes\mathbb{I}_n \subseteq^{E_0} \mn{n}$ is given by $\Phi(\mn{n}\otimes \mn{n})$.

        \item[(iii)] Let $e_1$ denote the Jones projection in the basic construction. Then
\[
e_1
=
\Phi \left( \frac{1}{n}\sum_{i=1}^n\sum_{j=1}^n E_{ij}\otimes E_{ij} \right),
\]
where $\{E_{ij}\}$ denotes the standard system of matrix units in $\mn{n}$.

        \item[(iv)] The dual conditional expectation $\Tilde{E}$ is given by
        \[
        \Tilde{E}=\Phi\left( \operatorname{id}\otimes \frac{1}{n}\operatorname{Tr} \right)\Phi^{-1},
        \]
        and, in particular,
        \begin{align*}
            A\otimes B
            &\mapsto A\otimes \frac{1}{n}\operatorname{Tr}(B)
            \in \mn{n}\otimes \C=\mn{n}.
        \end{align*}
    \end{itemize}
\end{theorem}
\begin{proof} Let $\tau:= \frac{1}{n} \operatorname{Tr}$. 

\noindent $(i)$. Indeed for any $X \in \mn{n}$
$$\sum_{i,j=1}^n u_{ij} E_0(u_{ij}^* X) = n \sum_{i,j=1}^n E_{ij} \tau(E_{ji} X)$$
Because the trace $\tau(E_{ji} X)$ simply extracts the normalized $(i, j)$-th entry of $X$, we have $\tau(E_{ji}X) = \frac{1}{n}X_{ij}$. Substituting this back
$$n \sum_{i,j=1}^n E_{ij} \left( \frac{1}{n} X_{ij} \right) = \sum_{i,j=1}^n X_{ij} E_{ij} = X.$$ 

Furthermore,
$$\text{Ind}(E_0) = \sum_{i,j=1}^n u_{ij} u_{ij}^* =n \sum_{i,j=1}^n E_{ij}E_{ji} = n \sum_{i=1}^n \left( \sum_{j=1}^n E_{ii} \right) = n \sum_{i=1}^n n E_{ii} = n^2 \sum_{i=1}^n E_{ii} = n^2 \mathbb{I}_n.$$

    \noindent $(ii)$. By Proposition \ref{prop:gupta-jones-proj}-(v), the $C^*$-basic construction is
    \[
    \sL_{\C \otimes \mathbb{I}_n} \left( (\mn{n}, \langle -,- \rangle_\tau) \right)
    \]
    which is equal to $\sB\left( \mn{n}, \langle -, - \rangle_\tau\right)$. It is easy to see that
\[
\mn{n} \otimes \mn{n}\cong^\Phi \sB\left( \mn{n}, \langle -, - \rangle_\tau\right).
\]

\noindent $(iii)$. By definition of the Jones' projection $e_1(X) = E_0(X)= \frac{1}{n}\mathrm{Tr}(X)\otimes \mathbb{I}_n$ for all $X \in \mn{n}$. Thus we get
    \[
    e_1(-) = \frac{1}{n} \operatorname{vec}(\mathbb{I}_n)^T \operatorname{vec}(-) \otimes \mathbb{I}_n 
    \]
    In the standard column-stacking vectorization, the basis $\{E_{kl}\}$ of $\mn{n}$ is ordered lexicographically by column index first:
    $$\mathcal{B}_{\text{col}} = \{E_{11}, E_{21}, \dots, E_{n1}, \, E_{12}, E_{22}, \dots, E_{n2}, \, \dots, \, E_{1n}, E_{2n}, \dots, E_{nn}\}.$$
    Under this ordered basis the matrix of $e_1$ is calculated as:
    $$
    [e_1]_{\mathcal{B}_{\text{col}}} = \frac{1}{n} \operatorname{vec}(\mathbb{I}_n) \operatorname{vec}(\mathbb{I}_n)^{T}= \frac{1}{n} \sum_{i=1}^n \sum_{j=1}^n E_{ij} \otimes E_{ij}
    $$

    \noindent $(iv)$. For any $X, Y \in \mn{n}$
    \begin{align*}
     \Tilde{E}(Xe_1Y) = \text{Ind}(E_0)^{-1}XY=\frac{1}{n^2} \mathbb{I}_n XY &= \frac{1}{n^2} X\mathbb{I}_{n}Y  \\
     &=  \frac{1}{n^2} X \left( \sum_{i=1}^n E_{ii} \right) Y\\
     &=  \frac{1}{n^2} \sum_{i,j=1}^n (X E_{ij} Y) \delta_{ij} \\
     &= \frac{1}{n} \sum_{i,j=1}^n (X E_{ij} Y) \cdot \left( \frac{1}{n}\text{Tr}(E_{ij}) \right)\\
     &= \Phi \circ \left( \text{id} \otimes \frac{1}{n}\text{Tr} \right) \left( \frac{1}{n} \sum_{i,j=1}^n (X E_{ij} Y) \otimes E_{ij} \right)\\
     &= \Phi \circ \left( \text{id} \otimes \frac{1}{n}\text{Tr} \right)\left((X \otimes \mathbb{I}_n) \left( \frac{1}{n} \sum_{i,j=1}^n E_{ij} \otimes E_{ij} \right) (Y \otimes \mathbb{I}_n)\right)\\
     &= \Phi \circ \left( \text{id} \otimes \frac{1}{n}\text{Tr} \right)\circ \Phi^{-1}(Xe_1Y)
    \end{align*}  
    This implies that $\Tilde{E}= \Phi \circ \left( \text{id} \otimes \frac{1}{n}\text{Tr} \right)\circ \Phi^{-1}.$
\end{proof}

\begin{remark}\label{rem:jones-proj-e-1}
    From now on, we suppress the isomorphism $\Phi$ from Theorem \ref{thm:basic-c-ten-I} and identify the basic construction associated with $\C\otimes \mathbb{I}_n \subseteq^{E_0} \mn{n}$ with $\mn{n}\otimes \mn{n}$. Under this identification, we regard the corresponding Jones projection as
    \[
   e_1= \frac{1}{n}\sum_{i,j=1}^n E_{ij}\otimes E_{ij}\in \mn{n}\otimes \mn{n},
    \]
    and the dual conditional expectation as
    \[
    \Tilde{E}=\text{id}\otimes \frac{1}{n}\text{Tr}=:E_n.
    \]
\end{remark}

In \cite{gupta2024}, the authors computed the Jones projection associated with the inclusion $u\Delta u^* \subseteq^{E_u} \mn{2}$. Using their computation, we derive an explicit formula for the interior angle between two masas induced by rotation matrices. We begin with the following lemma:

\begin{lem}\label{lem:jones-proj-r-theta}
   For an angle $\theta$, the Jones projection associated with the masa $r_\theta \Delta r_\theta^*$ is given by
\[
e^\theta :=
\begin{pmatrix}
a & b & b & c\\
b & c & c & -b\\
b & c & c & -b\\
c & -b & -b & a
\end{pmatrix},
\]
where
\[
a=1-\frac{1}{2}\sin^2(2\theta),\qquad
b=\frac{1}{4}\sin(4\theta),\qquad
c=\frac{1}{2}\sin^2(2\theta).
\]
\end{lem}
\begin{proof}
This follows from \cite[Lemma 4.3-(3)]{gupta2024}.
\end{proof}

\begin{theorem}
    For any two angles $\theta,\varphi$, we have
\[
\cos\left(\alpha\left(r_\theta\Delta r_\theta^*,r_\varphi\Delta r_\varphi^*\right)\right)
=
\left|\cos(2\theta-2\varphi)\right|.
\]
\end{theorem}
\begin{proof}
    Let $e^\theta$ and $e^\varphi$ denotes the Jones' projections corresponding to $r_\theta\Delta r_\theta^*$ and $r_\varphi\Delta r_\varphi^*$ respectively. Using Lemma \ref{lem:jones-proj-r-theta}, we calculate
    $$e^\theta e^\varphi = \begin{pmatrix}  x_{11} & x_{12} & x_{12} & x_{14} \\  x_{21} & x_{22} & x_{22} & -x_{21} \\  x_{21} & x_{22} & x_{22} & -x_{21} \\  x_{14} & -x_{12} & -x_{12} & x_{11}  \end{pmatrix}$$

    where the distinct $x_{ij}$'s as follows:
    $$x_{22} = \mathbf{\frac{1}{2} \sin(2\theta) \sin(2\varphi) \cos(2\theta - 2\varphi)}$$
    $$x_{12} = \mathbf{\frac{1}{2} \cos(2\theta) \sin(2\varphi) \cos(2\theta - 2\varphi)}$$$$x_{21} = \mathbf{\frac{1}{2} \sin(2\theta) \cos(2\varphi) \cos(2\theta - 2\varphi)}$$
    $$x_{11} = \mathbf{\frac{1}{4} \Big[ 1 + \cos^2(2\theta) + \cos^2(2\varphi) + \cos^2(2\theta - 2\varphi) \Big]}$$$$x_{14} = \mathbf{\frac{1}{4} \Big[ \sin^2(2\theta) + \sin^2(2\varphi) + \sin^2(2\theta - 2\varphi) \Big]}$$

    Therefore,
    \begin{align*}
        M(\theta, \varphi):=E_2(e^\theta e^\varphi) - E_2\left(e_1 \right)& = \begin{pmatrix}
        \frac{1 + \cos^2(2\theta - 2\varphi)}{4} &  -\frac{1}{8} \sin(4\theta - 4\varphi)\\
        \frac{1}{8} \sin(4\theta - 4\varphi) & \frac{1 + \cos^2(2\theta - 2\varphi)}{4}
    \end{pmatrix}
    -
    \frac{1}{4}\begin{pmatrix}
       1&  0\\
       0 & 1
    \end{pmatrix}\\
    &=\begin{pmatrix}
        \frac{\cos^2(2\theta - 2\varphi)}{4} &  -\frac{1}{8} \sin(4\theta - 4\varphi)\\
        \frac{1}{8} \sin(4\theta - 4\varphi) & \frac{\cos^2(2\theta - 2\varphi)}{4}
    \end{pmatrix}
    \end{align*}
    \noindent\textbf{Claim.} $\|M(\theta, \varphi)\|_{\text{op}}=\frac{1}{4}|\cos(2\theta - 2 \varphi)|$

    The matrix $M(\theta, \varphi)$ has the form $\begin{pmatrix} a & -b \\ b & a \end{pmatrix}$ for $a:=\frac{\cos^2(2\theta - 2\varphi)}{4} , b:=\frac{1}{8} \sin(4\theta - 4\varphi) \in \R$.
    
    For any matrix $A = \begin{pmatrix} a & -b \\ b & a \end{pmatrix}$, the operator norm is given by the square root of the largest eigenvalue of $A^T A$.
    $$A^T A = \begin{pmatrix} a & b \\ -b & a \end{pmatrix} \begin{pmatrix} a & -b \\ b & a \end{pmatrix} = \begin{pmatrix} a^2 + b^2 & 0 \\ 0 & a^2 + b^2 \end{pmatrix}$$
    
    The operator norm is therefore exactly $\sqrt{a^2 + b^2}=\frac{1}{4}|\cos(2\theta - 2 \varphi)|$.

    Therefore, 
    \[
    \cos\left( \alpha\left( r_\theta \Delta r_\theta^* , r_\varphi \Delta r_\varphi^* \right) \right) = \frac{\|M(\theta, \varphi)\|_{\text{op}}}{\|e^\theta - e_1\| \|e^\varphi - e_1\|}= \frac{\|M(\theta, \varphi)\|_{\text{op}}}{1/4}= |\cos(2\theta - 2 \varphi)|.
    \]
\end{proof}

\begin{prop}\label{prop:jones-proj}
    Let $(E_{ij})$ be the standard matrix units in $\mn{n}$, and let $u\in\sU(\mn{n})$. For $1\leq k\leq n$, let 
$p_k=uE_{kk}u^*$, be the orthogonal projection onto the $k$-th column of $u$. If $e^u$ denotes the Jones projection in the basic construction associated with the inclusion
$u\Delta^{(n)}u^*\subseteq^{E_u}\mn{n},$ (see (\ref{eq:unitary-conditional}))
then
\[
e^u
=
\sum_{k=1}^n p_k\otimes\overline{p_k}
=
\sum_{k=1}^n
\left(uE_{kk}u^*\right)
\otimes
\left(\overline{u}E_{kk}u^T\right),
\]
where $\overline{p_k}$ denotes the entrywise complex conjugate of $p_k$.
\end{prop}
\begin{proof}
    First, consider the diagonal subalgebra $\Delta^{(n)} \subseteq M:=\mn{n}$. The canonical trace-preserving conditional expectation $E_{\Delta^{(n)}}: M \to \Delta^{(n)}$ simply isolates the diagonal entries of a matrix:
    $$E_{\Delta^{(n)}}(x) = E_{11} x E_{11} + \cdots + E_{nn} x E_{nn}= \sum_{k=1}^n E_{kk} x E_{kk}$$

    Now, our subalgebra is $N = u\Delta^{(n)} u^*$. Because $u$ is a unitary matrix, the map $\text{Ad}_u(y) = u y u^*$ is an algebra automorphism of $M$ that preserves the trace $\tau$. Thus for any $x \in M$:
    \begin{align*}
        E_N(x) = u E_{\Delta^{(n)}}(u^* x u) u^* &= u \Big( \sum_{k=1}^n E_{kk} (u^*xu) E_{kk} \Big) u^* = \sum_{k=1}^n (uE_{kk}u^*)x (uE_{kk}u^*)=\sum_{k=1}^n p_k xp_k
    \end{align*}
    Therefore, 
    $$e^u(x) = E_N(x) = \sum_{k=1}^n p_k x p_k \quad (x \in M).$$

    The basic construction $\langle M, e_N \rangle$ acts on $L^2(M)$. The algebra of all linear operators on $\mn{n}$ is canonically isomorphic to $\mn{n} \otimes \mn{n}$. 

    Under the standard Kronecker tensor product convention (or vectorization isomorphism), the basic tensor $A \otimes B$ acts on a matrix $x \in \mn{n}$ via left multiplication by $A$ and right multiplication by the transpose $B^T$:
    $$(A \otimes B)(x) = A x B^T$$

    Therefore, the tensor representation of the operator $e^u$ is:
    $$e^u = \sum_{k=1}^n p_k \otimes p_k^T$$

    By definition, $p_k = u E_{kk} u^*$. Because $E_{kk}$ is a real, symmetric, idempotent matrix, $p_k$ is an orthogonal projection, which means it is self-adjoint. Hence $p_k = p_k^*  = \overline{p_k}^T, \text{ that is, } p_k^T = \overline{p_k}.$
    Therefore, 
    $$e^u = \sum_{k=1}^n p_k \otimes \overline{p_k}$$
    This completes the proof.
\end{proof}

\begin{cor}\label{cor:equval-of-jones-proj}
    For a $n\times n$ unitary matrix $u\in\mn{n}$, we have
\[
e^u=(u\otimes\overline{u})e^{\mathbb{I}_n}(u\otimes\overline{u})^*.
\]
Consequently, the Jones projections associated with any two masas of $\mn{n}$ are unitarily equivalent in $\mn{n}\otimes\mn{n}$ to the projection $e^{\mathbb{I}_n}$.
\end{cor}
\begin{proof}
    Indeed from Proposition \ref{prop:jones-proj}, we get
    \begin{align*}
        e^u = \sum_{k=1}^n \left( uE_{kk} u^* \right) \otimes \left(\overline{u} E_{kk} u^T\right)=  (u\otimes \overline{u}) \left( \sum_{k} (E_{kk} \otimes E_{kk}) \right) (u^* \otimes u^T) = (u\otimes \overline{u}) e^{\mathbb{I}_n} (u^* \otimes u^T).
    \end{align*}
    To complete the proof note that $u^T = \overline{u}^*$.
\end{proof}

\begin{cor}\label{cor:equval-jones-2}
    For any two $n\times n$ unitary matrices $u,v\in\mn{n}$, we have
\[
e^v
=
(u\otimes\overline{u})\,e^{u^*v}\,(u\otimes\overline{u})^*.
\]
\end{cor}
\begin{proof}
    Write $v=u(u^*v)$. Using the identity
\[
(AB)\otimes\overline{AB}=(A\otimes\overline{A})(B\otimes\overline{B}),
\]
we obtain
\[
v\otimes\overline{v}
=
(u\otimes\overline{u})
\big((u^*v)\otimes\overline{u^*v}\big).
\]

Now, by Corollary \ref{cor:equval-of-jones-proj},
\[
\begin{aligned}
e^v
&=(v\otimes\overline{v})e^{\mathbb{I}_n}(v\otimes\overline{v})^*\\
&=(u\otimes\overline{u})
\Big((u^*v)\otimes\overline{u^*v}\Big)
e^{\mathbb{I}_n}
\Big((u^*v)\otimes\overline{u^*v}\Big)^*
(u\otimes\overline{u})^*\\
&=(u\otimes\overline{u})e^{u^*v}(u\otimes\overline{u})^*.
\end{aligned}
\]
\end{proof}

\begin{lem}\label{lem:partial_trace_property}
    For any unitary $u\in\sU(\mn{n})$ and any $X\in\mn{n}\otimes\mn{n}$, we have
\[
\left(\mathrm{id}\otimes\frac{1}{n}\operatorname{Tr}\right)
\Big((u\otimes\bar{u})X(u\otimes\bar{u})^*\Big)
=
(u\otimes 1)
\left[
\left(\mathrm{id}\otimes\frac{1}{n}\operatorname{Tr}\right)(X)
\right]
(u\otimes 1)^*.
\]
\end{lem}
\begin{proof}
    The identity can be verified for simple tensors and then by linearity the proof follows.
\end{proof}

\begin{prop}
   Let $u,v\in\sU(\mn{n})$ be arbitrary $n \times n$ unitary matrices. Then the product $e^u e^v$ is unitarily equivalent to $e^{\mathbb{I}_n}e^{u^*v}$ via the unitary $u\otimes\bar{u}$; that is,
\[
e^u e^v
=
(u\otimes\bar{u})
\left(e^{\mathbb{I}_n}e^{u^*v}\right)
(u\otimes\bar{u})^*.
\]
\end{prop}
\begin{proof}
    Indeed using corollary \ref{cor:equval-of-jones-proj} and corollary \ref{cor:equval-jones-2} both we get, 
    $$\begin{aligned}   
    e^u e^v &= \Big[ (u \otimes \bar{u}) e^{\mathbb{I}_n} (u \otimes \bar{u})^* \Big] \Big[ (u \otimes \bar{u}) e^{u^* v} (u \otimes \bar{u})^* \Big] \\    &= (u \otimes \bar{u}) e^{\mathbb{I}_n} \Big( (u \otimes \bar{u})^* (u \otimes \bar{u}) \Big) e^{u^* v} (u \otimes \bar{u})^*   
    \end{aligned}$$

    Since $(u \otimes \bar{u})$ is unitary, $(u \otimes \bar{u})^* (u \otimes \bar{u}) = \mathbb{I}_{n^2}$:
    $$e^u e^v = (u \otimes \bar{u}) \left( e^{\mathbb{I}_n} e^{u^* v} \right) (u \otimes \bar{u})^*$$
\end{proof}

\begin{lem}\label{lem:eu-e-1}
For any $n\times n$ unitary matrices $u\in\mn{n}$, we have
\[
\|e^u-e_1\|_{\mathrm{op}}=\frac{\sqrt{n-1}}{n}.
\]
\end{lem}
\begin{proof}
First note that $E_n(e_1)= \frac{1}{n^2} \mathbb{I}_n\otimes 1$ for all $n\ge 2$. Indeed,
\begin{align*}
    E_n(e_1) = \left( \operatorname{id} \otimes \frac{1}{n} \operatorname{Tr} \right)\left(\frac{1}{n} \sum_{i=1}^n \sum_{j=1}^n E_{ij} \otimes E_{ij}\right)
    =\frac{1}{n^2}\sum_{i, j}\left(E_{ij} \otimes \operatorname{Tr}(E_{ij})\right)= \frac{1}{n^2} \mathbb{I}_n\otimes 1.
\end{align*}

Also from Proposition \ref{prop:jones-proj}, 
\begin{align*}
    E_n(e^u) = \left(\text{id} \otimes \frac{1}{n}\text{Tr}\right)(e^u) 
    &= \left(\text{id} \otimes \frac{1}{n}\text{Tr}\right) \left( \sum_{k=1}^n p_k \otimes \overline{p_k} \right)\\
    &= \sum_{k = 1}^n p_k \otimes \frac{1}{n} \operatorname{Tr}(u E_{kk} u^*)
    = \sum_{k = 1}^n p_k \otimes \frac{1}{n} =  \left(\sum_{k = 1}^n u E_{kk} u^*\right) \otimes \frac{1}{n} = \frac{1}{n}\mathbb{I}_n \otimes 1
\end{align*}

    Therefore, 
    \begin{align*}
       E_n(e^u-e_1)= \left(\frac{1}{n} \mathbb{I}_n\otimes 1 \right)- \left(\frac{1}{n^2}\mathbb{I}_n \otimes 1\right)=\left( \frac{n-1}{n^2}\right) \mathbb{I}_n \otimes 1.
    \end{align*}
    
    Hence we get, 
    \begin{align*}
    \|e^u - e_1\|^2 
    &= \|E_n\left((e^u -e_1)(e^u -e_1)\right)\|_{\text{op}}= \|E_n\left(e^u -e_1\right)\|_{\text{op}} = \frac{n-1}{n^2}.
    \end{align*}
    
\end{proof}

\begin{lem}\label{lem:angle-formula}
   For $u,v\in\sU(\mn{n})$, define $
M(u,v):=E_n(e^u e^v)-E_n(e_1).$
Then
\[
\cos\alpha\left(u\Delta^{(n)}u^*,v\Delta^{(n)}v^*\right)
=
\frac{n^2\|M(u,v)\|_{\mathrm{op}}}{n-1}.
\]
\end{lem}
\begin{proof}
By the Defintion \ref{def:interior_angle}, we get
    \[
    \cos \alpha(u\Delta^{(n)} u^*, v\Delta^{(n)} v^*)
    = \frac{\|\langle e^u - e_1, e^v -e_1 \rangle \|_{\text{op}}}{\|e^u - e_1 \|_{\text{op}} \|e^v - e_1 \|_{\text{op}}}
    = \frac{\|M(u, v) \|_{\text{op}}}{\|e^u - e_1 \|_{\text{op}} \|e^v - e_1 \|_{\text{op}}}
    \]
and rest follows from Proposition \ref{prop:gupta-jones-proj}-(iii) and Lemma \ref{lem:eu-e-1}.
\end{proof}

\begin{lem}\label{lem:prop-partial-tr}
  Let $\tau:\mn{k}\to\C$ denote the normalized trace on $\mn{k}$, and let
\[
\operatorname{id}\otimes\tau:\mn{n}\otimes\mn{k}\to\mn{n}\otimes\C
\]
be the partial trace defined by
\[
(\operatorname{id}\otimes\tau)(A\otimes B):=A\otimes\tau(B).
\]
Then, for all $A,C\in\mn{n}$, every unitary $U\in\mn{k}$, and every $X\in\mn{n}\otimes\mn{k}$, we have
\[
(\operatorname{id}\otimes\tau)
\Big((A\otimes U)X(C\otimes U^*)\Big)
=
(A\otimes1)
\Big[(\operatorname{id}\otimes\tau)(X)\Big]
(C\otimes1).
\]
\end{lem}
\begin{proof}
Since $\mn{n}\otimes\mn{k}$ is spanned by simple tensors and $\operatorname{id}\otimes\tau$ is linear, it suffices to verify the identity for a simple tensor $X=Y\otimes Z$, where $Y\in\mn{n}$ and $Z\in\mn{k}$.

For the left-hand side, we have
\[
(A\otimes U)(Y\otimes Z)(C\otimes U^*)
=
(AYC)\otimes(UZU^*).
\]
Applying $\operatorname{id}\otimes\tau$ and using the unitary invariance of the normalized trace, we obtain
\[
\begin{aligned}
(\operatorname{id}\otimes\tau)
\Big((A\otimes U)(Y\otimes Z)(C\otimes U^*)\Big)
&=(AYC)\otimes\tau(UZU^*)\\
&=(AYC)\otimes\tau(Z).
\end{aligned}
\]

On the other hand,
\[
(\operatorname{id}\otimes\tau)(Y\otimes Z)
=
Y\otimes\tau(Z),
\]
and hence
\[
\begin{aligned}
(A\otimes1)
\Big[(\operatorname{id}\otimes\tau)(Y\otimes Z)\Big]
(C\otimes1)
&=(A\otimes1)(Y\otimes\tau(Z))(C\otimes1)\\
&=(AYC)\otimes\tau(Z).
\end{aligned}
\]
Thus both sides agree on every simple tensor, and therefore, by linearity, the desired identity holds for all $X\in\mn{n}\otimes\mn{k}$.
\end{proof}
\begin{lem}\label{lem:uni-inv-muv}
   For unitaries $u,v,W\in\mn{n}$, we have
\[
M(Wu,Wv)
=
(W\otimes1)M(u,v)(W^*\otimes1).
\]
In particular,
\[
M(\mathbb{I}_2,u^*v)
=
M(u^*u,u^*v)
=
(u^*\otimes1)M(u,v)(u\otimes1).
\]
\end{lem}
\begin{proof}
    The new column projections for $u$ are:
    $$p'_k = (Wu) E_{kk} (Wu)^* = W (u E_{kk} u^*) W^* = W p_k W^*$$
    
    The complex conjugate transforms as $\overline{p'_k} = \overline{W} \overline{p_k} \overline{W}^*$.
    
    Substitute these into the tensor formula for $e^{Wu}$:
    $$e^{Wu} = \sum_{k=1}^n (W p_k W^*) \otimes (\overline{W} \overline{p_k} \overline{W}^*) = (W \otimes \overline{W}) e^u (W^* \otimes \overline{W}^*)$$
    
    Because $(W^* \otimes \overline{W}^*)(W \otimes \overline{W}) = \mathbb{I}_{n^2}$, multiplying $e^{Wu}$ and $e^{Wv}$ gives:
    $$e^{Wu} e^{Wv} = (W \otimes \overline{W}) e^u e^v (W^* \otimes \overline{W}^*)$$
    
    Now, applying Lemma \ref{lem:prop-partial-tr} on the partial trace $E_n$:
    \begin{align*}
        E_n(e^{Wu} e^{Wv}) &= (\text{id} \otimes \tau)\Big[ (W \otimes \overline{W}) (e^u e^v) (W^* \otimes \overline{W}^*) \Big] \\
        &=(W\otimes 1) \Big[ (\text{id} \otimes \tau)(e^u e^v) \Big] (W^*\otimes 1) = (W\otimes 1) E_n(e^u e^v) (W^*\otimes 1)
    \end{align*}
    
    Subtracting $E_n(e_1)$, we get $M(Wu, Wv) = (W\otimes 1) M(u, v) (W^*\otimes 1)$.
\end{proof}

\begin{theorem}[Unitary Invariance of Interior Angle]\label{thm:uni-inv-angle}
Let $u,v\in\sU(\mn{n})$ be arbitrary unitary matrices. Then
\[
\alpha\left(u\Delta^{(n)}u^*,v\Delta^{(n)}v^*\right)
=
\alpha\left(\Delta^{(n)},u^*v\Delta^{(n)}v^*u\right).
\]
\end{theorem}
\begin{proof}
Indeed, by Lemma \ref{lem:angle-formula} and Lemma \ref{lem:uni-inv-muv},

    \begin{align*}
        \cos \alpha(u\Delta^{(n)} u^*, v\Delta^{(n)} v^*) &=\frac{n^2}{n-1} \Vert{}M(u,v)\Vert{}_{\mathrm{op}} \\
        &= \frac{n^2}{n-1} \Vert{}(u^*\otimes 1)M(u,v) (u\otimes 1)\Vert{}_{\mathrm{op}} \\
        &= \frac{n^2}{n-1} \Vert{}M(\mathbb{I}_n, u^*v)\Vert{}_{\mathrm{op}} 
        =\cos \alpha(\Delta^{(n)}, u^*v \Delta^{(n)} v^*u).
    \end{align*}
\end{proof}

\noindent Using this above result, we can generalise Corollary \ref{cor:relation_angle_distance} as follows:
\begin{theorem}\label{thm:dkk-is-sin}
    Let $u, v$ be two $2\times 2$ unitary matrices. Then
    \[
        \dkk(u\Delta u^*, v\Delta v^*) = \sin \alpha(u\Delta u^*, v\Delta v^*).
    \]
\end{theorem}
\begin{proof}
    Follows immediately from Theorem \ref{prop:dkk-gen-masa}, Corollary \ref{cor:relation_angle_distance} and Theorem \ref{thm:uni-inv-angle}.
\end{proof}

\begin{cor}\label{cor:dkk-is-sin-of-angle}
    For any two intermediate subalgebras $\C\otimes \mathbb{I}_2 \subsetneq\sA, \sB\subsetneq \mn{2}$, we have
    \[
    \dkk(\sA, \sB) = \sin \alpha(\sA, \sB).
    \]
\end{cor}

\begin{proof}
    This follows from Theorem \ref{thm:dkk-is-sin}, since every such unital $C^*$-subalgebra of $\C \otimes \mathbb{I}_2 \subseteq \mn{2}$ is, up to unitary conjugation, a masa.
\end{proof}

\begin{prop}\label{prop:muv-calc}
    For $u,v\in\sU(\mn{n})$, let $M(u,v)$ be as in Lemma \ref{lem:angle-formula}. Then
\[
M(u,v)
=
\left[
\frac{1}{n}\,
u\left(
|u^*v|^{\circ 2}\circ(u^*v)
\right)v^*
-\frac{1}{n^2}\mathbb{I}_n
\right]\otimes1.
\]
where $\circ$ denotes the Hadamard product between two matrices.
\end{prop}
\begin{proof}
    First note that using Corollary \ref{cor:equval-jones-2} and  Lemma \ref{lem:partial_trace_property},
    $$E_n(e^u e^v) = \left(\mathrm{id} \otimes \frac{1}{n}\operatorname{Tr}\right)\Big((u \otimes \bar{u}) \left( e^{\mathbb{I}_n} e^{u^* v} \right) (u \otimes \bar{u})^*\Big) = (u\otimes 1) \left[ \left(\mathrm{id} \otimes \frac{1}{n}\operatorname{Tr}\right)\left( e^{\mathbb{I}_n} e^{u^* v} \right) \right] (u\otimes 1)^*$$

    Write $w:=u^*v$. We get using Proposition \ref{prop:jones-proj}
    $$e^{\mathbb{I}_n} e^w = \sum_{k=1}^n \sum_{l=1}^n \Big( E_{kk} w E_{ll} w^* \Big) \otimes \Big( E_{kk} \bar{w} E_{ll} w^T \Big)$$

    Now applying the linear map $(\mathrm{id} \otimes \frac{1}{n}\operatorname{Tr})$:
    $$\left(\mathrm{id} \otimes \frac{1}{n}\operatorname{Tr}\right)\left( e^{\mathbb{I}_n} e^w \right) = \frac{1}{n} \sum_{k=1}^n \sum_{l=1}^n \Big( E_{kk} w E_{ll} w^* \Big) \otimes \operatorname{Tr}\Big( E_{kk} \bar{w} E_{ll} w^T \Big)$$

    Using the tracial property of the trace on the scalar coefficient:
    $$
    \operatorname{Tr}\Big( E_{kk} \bar{w} E_{ll} w^T \Big) = \operatorname{Tr}\Big( (E_{kk} \bar{w} E_{ll}) (w^T E_{kk}) \Big) = \bar{w}_{kl} \cdot (w^T)_{lk} = \bar{w}_{kl} \cdot w_{kl} = \vert{}w_{kl}\vert{}^2
    $$

    Therefore, 
    $$
    \begin{aligned} \left(\mathrm{id} \otimes \frac{1}{n}\operatorname{Tr}\right)\left( e^{\mathbb{I}_n} e^w \right) &= \frac{1}{n} \sum_{k=1}^n \sum_{l=1}^n \left(\vert{}w_{kl}\vert{}^2 \Big( E_{kk} w E_{ll} w^* \Big)\otimes 1\right) \\ &= \frac{1}{n} \left( \sum_{k=1}^n \sum_{l=1}^n \vert{}w_{kl}\vert{}^2 E_{kk} w E_{ll} \right) w^* \otimes 1 \end{aligned}
    $$
    Since $E_{kk} w E_{ll} = w_{kl} E_{kl}$:
    $$
    \begin{aligned} E_n(e^{\mathbb{I}_n} e^{u^*v} )=\left(\mathrm{id} \otimes \frac{1}{n}\operatorname{Tr}\right)\left( e^{\mathbb{I}_n} e^{w} \right) 
    &= \frac{1}{n} \left( \sum_{k=1}^n \sum_{l=1}^n \vert{}w_{kl}\vert{}^2 w_{kl} E_{kl} \right) w^* \otimes 1= \frac{1}{n} \Big( \vert{}w\vert{}^{\circ 2} \circ w \Big) w^* \otimes 1
    \end{aligned}
    $$
    where $\circ$ denotes the Hadamard (entrywise) product.

    Thus we get:
    \begin{align*}
        E_n(e^u e^v) = (u\otimes 1) E_n(e^{\mathbb{I}_n} e^{u^*v} ) (u\otimes 1)^* =\frac{1}{n}  u \left( |u^*v|^{\circ 2} \circ (u^*v)\right) v^* \otimes 1.
    \end{align*}
    Hence,
   \begin{align*}
     M(u, v) = E_{n}(e^ue^v)-E_{n}(e_1)& = 
     \left[\frac{1}{n}  u \left( |u^*v|^{\circ 2} \circ (u^*v)\right) v^* \otimes 1\right] - \left[\frac{1}{n^2} \mathbb{I}_n \otimes 1\right]\\
     &= \left[\frac{1}{n}  u \left( |u^*v|^{\circ 2} \circ (u^*v)\right) v^*
    - \frac{1}{n^2} \mathbb{I}_n \right] \otimes 1.   
   \end{align*}
\end{proof}

\begin{theorem}\label{thm:angle-betwn-masas}
   Let $u,v\in\mn{n}$ be two $n\times n$ unitary matrices. Then
\[
\cos\alpha\left(u\Delta^{(n)}u^*,v\Delta^{(n)}v^*\right)
=
\frac{n}{n-1}
\left\|
\left[
|u^*v|^{\circ 2}\circ(u^*v)
-\frac{1}{n}u^*v
\right]
\right\|_{\mathrm{op}}.
\]
\end{theorem}
\begin{proof}
    Indeed using Proposition \ref{prop:muv-calc}, 
    \begin{align*}
\|M(u, v)\|_{\operatorname{op}} &=  \left\| \left[\frac{1}{n}  u \left( |u^*v|^{\circ 2} \circ (u^*v)\right) v^*
    - \frac{1}{n^2} \mathbb{I}_n \right] \otimes 1\right\|_{\operatorname{op}}\\
 & = \frac{1}{n}\left\|u\left[\left(|u^*v|^{\circ 2} \circ (u^*v)\right) - \frac{1}{n} u^*v \right]v^*\right\|_{\operatorname{op}}\\
 & =  \frac{1}{n}\left\|\left[\left(|u^*v|^{\circ 2} \circ (u^*v)\right) - \frac{1}{n} u^*v \right]\right\|_{\operatorname{op}}
\end{align*}
Rest follows from Lemma \ref{lem:angle-formula}.
\end{proof}

\begin{remark}\label{rem:angle-not-good}
It was shown in \cite[Proposition 2.3]{bakshi-etal-2019} that
$\alpha(\sP,\sQ)=0$ if and only if $\sP=\sQ$ for any two $II_1$-subfactors of a given factor.
This result supports the intuition that $\alpha$ behaves like a genuine
notion of ``angle.'' The same phenomenon occurs for the angle between
$\Delta$ and $u\Delta u^*$ in $\mn{2}$; see
\cite[Corollary 4.5-(2)]{gupta2024}.

However, this intuition fails in general for subalgebras of
$C^*$-algebras. To illustrate this, consider the unitary matrix
\[
    u =
    \begin{pmatrix}
        1 & 0 & 0 \\
        0 & \frac{\sqrt{3}}{2} & \frac{1}{2} \\
        0 & -\frac{1}{2} & \frac{\sqrt{3}}{2}
    \end{pmatrix}.
\]
Using the formula in Theorem~\ref{thm:angle-betwn-masas}, we obtain
\[
    \alpha\left(\Delta^{(3)},u\Delta^{(3)}u^*\right)
    = \cos^{-1}(1) = 0.
\]
Thus, even though $\Delta^{(3)}\neq u\Delta^{(3)}u^*$, their angle is
zero. Moreover, the same kind of phenomenon occurs in every $\mn{n}$ for $n\geq 3$. In Corollary \ref{cor:angle-not-inv-sim}, we provide another example of this phenomenon for a quadruple of simple $C^*$-algebras.
\end{remark}

\begin{lem}\label{lem:eig-sp-swap}
    Let $a, b \in \mathbb{C}$ and let $S:=\sum_{i,j}E_{ij} \otimes E_{ji} \in \mn{n^2}$ denote the swap operator acting on $\mathbb{C}^n \otimes \mathbb{C}^n$. Then
    $$\Vert{}a\mathbb{I}_{n^2} + bS\Vert{}_{\mathrm{op}} = \max\{\vert{}a+b\vert{}, \vert{}a-b\vert{}\}$$
\end{lem}

\begin{lem}\label{lem:u-theta}
    Let $n \ge 2$. Consider the ambient $C^*$-algebra $\sA=\mn{n^2} \cong \mn{n} \otimes \mn{n}$ equipped with the canonical conditional expectation $\frac{1}{n^2}\operatorname{Tr}(-)\mathbb{I}_{n^2}$ onto $\sB=\mathbb{C}\mathbb{I}_{n^2}$ and $\sC = \mn{n}\otimes \C\mathbb{I}_n$.  Let $S \in \mn{n^2}$ be the swap operator defined by $S(u \otimes v) = v \otimes u$. For any angle $\theta \in \mathbb{R}$, define the unitary operator $U_\theta = (\cos\theta) \mathbb{I}_{n^2} + i(\sin\theta )S$. Then 
    \begin{itemize}
        \item[(i)] $U_\theta \sC U_\theta^* \in \operatorname{IMS}(\sB, \sA, E)$ for all $\theta$.
        \item[(ii)] $\operatorname{Ind}(E) = n^4 \mathbb{I}_{n^2}$.
        \item[(iii)]  $\{\lambda_{ab}(\theta) = \sqrt{n} U_\theta (E_{ab} \otimes \mathbb{I}_n) U_\theta^*:\,1 \le a, b \le n\}$ forms a valid quasibasis for $E_{\vert{}U_\theta \sC U_\theta^*}$, and its Watatani index evaluates to the scalar matrix $\operatorname{Ind}(E_{\vert{}U_\theta \sC U_\theta^*}) = n^2 \mathbb{I}_{n^2}$.
        \item[(iv)] Let $\mu_{ab}:=\lambda_{ab}(0)$ for $a, b = 1, \dots, n.$. Then for all $a, b, c, d = 1, \dots, n$,
        \[
        E\left( \mu_{ab}^* \lambda_{cd}(\theta) \right) = \left( 
        \cos^2\theta \delta_{ac} \delta_{bd} +
        \frac{\sin^2 \theta}{n} \delta_{ba} \delta_{cd} 
        \right)\mathbb{I}_{n^2}
        \]
        where $\delta_{rs}$ is the Kronecker $\delta$-function.
    \end{itemize}
\end{lem}

\begin{prop}\label{prop:angle-left-tensor}
    For $\sC$ and $U_\theta$ are as defined in Lemma \ref{lem:u-theta} we have,
    $$\cos\alpha(\sC, U_\theta \sC U_\theta^*) = \frac{\cos^2\theta \sqrt{1 + (n^2-2n)\cos^2\theta}}{n-1}$$
\end{prop}
\begin{proof}
Using \cite[Proposition 3.3-(1)]{gupta2024} we have,
    $$
    \cos\alpha(\sC, U_\theta \sC U_\theta^*) = \frac{\| \sum_{a,b,c,d = 1}^n \mu_{ab} E\left( \mu_{ab}^* \lambda_{cd}(\theta) \right) \lambda_{cd}(\theta)^* - \mathbb{I}_{n^2}\|_{\mathrm{op}}}{\|\operatorname{Ind}(E_{\vert{}U_\theta \sC U_\theta^*} )\|}
    $$

    Let $T:=\sum_{a,b,c,d = 1}^n \mu_{ab} E\left( \mu_{ab}^* \lambda_{cd}(\theta) \right) \lambda_{cd}(\theta)^*$. Using Lemma \ref{lem:u-theta}-(iv) we can write $T = T_1 + T_2$ where
    \begin{align*}
        T_1 = \sum_{a,b,c,d} \cos^2\theta \delta_{ac} \delta_{bd} \mu_{ab} \lambda_{cd}(\theta)^*= n \cos^2 \theta \left( \sum_{a,b} (E_{ab} \otimes \mathbb{I}_n) U_\theta (E_{ba} \otimes \mathbb{I}_n)  \right) U_\theta^*
    \end{align*}
    here the middle term boils down to 
    \[
    \cos\theta \sum_{a,b} (E_{ab} E_{ba} \otimes \mathbb{I}_n) + i \sin\theta \sum_{a,b} (E_{ab} \otimes \mathbb{I}_n) S ( E_{ba} \otimes \mathbb{I}_n) = (n \cos \theta + i \sin\theta)  \mathbb{I}_{n^2}
    \]
     and thus 
     $T_1 = (n^2 \cos^4 \theta + i n \cos^3 \theta \sin\theta) \mathbb{I}_{n^2} + (n\cos^2 \theta \sin^2\theta - in^2 \cos^3 \theta \sin \theta) S$
     and 
     \[
     T_2 = \frac{\sin^2\theta}{n} \sum_{a,b,c,d} \delta_{ba}\delta_{cd}\mu_{ab} \lambda_{cd}(\theta)^* = \frac{\sin^2\theta}{n} \left(\sum_a \mu_{aa}\right) \left( \sum_c \lambda_{cc}(\theta)^* \right)= \frac{\sin^2\theta}{n} (\sqrt{n}\mathbb{I}_{n^2}) (\sqrt{n}\mathbb{I}_{n^2})= (\sin^2 \theta) \mathbb{I}_{n^2}
     \]
     Therefore we get 
     \begin{align*}
         T-\mathbb{I}_{n^2} = \cos^2 \theta (a\mathbb{I}_{n^2} + bS)
     \end{align*}
     where $a := (n^2 \cos^2 \theta - 1) + i n \cos\theta \sin\theta$ and $b:=n\sin^2\theta - i n^2 \cos\theta\sin\theta$.

     The eigenvalues of this matrix are $\lambda_{\pm} = a \pm b$ (by Lemma \ref{lem:eig-sp-swap}). Taking the absolute square of both:$$\vert{}\lambda_+\vert{}^2 = (n-1)^2 \left[ (n^2+2n)\cos^2\theta + 1 \right]$$
$$\vert{}\lambda_-\vert{}^2 = (n+1)^2 \left[ (n^2-2n)\cos^2\theta + 1 \right]$$
     Subtracting the two reveals $\vert{}\lambda_+\vert{}^2 - \vert{}\lambda_-\vert{}^2 = -4n\sin^2\theta \le 0$, meaning the maximum eigenvalue is always $\lambda_-$. Hence
     \[
     \|T-\mathbb{I}_{n^2}\|_{\operatorname{op}}= |\lambda_-|=\sqrt{(n+1)^2 \left[ (n^2-2n)\cos^2\theta + 1 \right]}.
     \]
     To complete the proof note that from Lemma \ref{lem:u-theta}-(iii), $\|\operatorname{Ind}(E_{\vert{}U_\theta \sC U_\theta^*} )\|= n^2-1$.
\end{proof}

In \cite[Remark 3.4]{gupta2024}, the authors raised the question of whether the vanishing of the interior angle between intermediate $C^*$-subalgebras necessarily implies their equality when the inclusion is not irreducible. The following consequence of Proposition \ref{prop:angle-left-tensor} answers this question in the negative.

\begin{cor}\label{cor:angle-not-inv-sim}
The implication
\[
\alpha(\sC,\sD)=0 \implies \sC=\sD
\]
fails for simple intermediate $C^*$-subalgebras when the inclusion $\sB\subseteq\sA$ is not irreducible. Indeed, for $n\ge 3$, let
\[
\sA=\mn{n^2},\qquad
\sB=\mathbb{C}\mathbb{I}_{n^2},\qquad
\sC=\mn{n}\otimes\mathbb{C}\mathbb{I}_n,
\]
and let $S$ be the flip operator on $\mathbb{C}^n\otimes\mathbb{C}^n$. Then
\[
\alpha(\sC,S\sC S^*)=0,
\qquad
\sC\neq S\sC S^*.
\]
Thus, the question raised in \cite[Remark 3.4]{gupta2024} has a negative answer.
\end{cor}

Recall that, under the additional assumption that $\sB\subseteq\sA$ is an irreducible inclusion of simple $C^*$-algebras, \cite[Proposition 5.10-(i)]{Bakshi_gupta_2021} establishes the equivalence
\[
    \alpha(\sC,\sD)=0
    \quad\Longleftrightarrow\quad
    \sC=\sD
\]
for simple intermediate $C^*$-subalgebras $\sC,\sD\in\operatorname{IMS}(\sB,\sA,E)$. Thus, the irreducibility hypothesis cannot, in general, be omitted.

\begin{definition}[Flat Unitary]
Let $u = (u_{jk})_{j,k=1}^n \in \mn{n}$ be a unitary matrix, and let $\Vec{u}_k$ denote the $k$-th column vector of $u$ for each $k \in \{1, \dots, n\}$. We say that $u$ is a \emph{flat unitary} if every non-zero entry within a given column has the exact same magnitude. Equivalently, $u$ is a flat unitary if for all $j, k \in \{1, \dots, n\}$,$$\vert{}u_{jk}\vert{} \in \left\{ 0, \frac{1}{\sqrt{\|\Vec{u}_k\|_0}} \right\},$$
where $\|\Vec{u}_k\|_0$ is the Hamming number (see Definition \ref{def:hamming-num}) of the column vector $\Vec{u}_k$.
\end{definition}

\begin{prop}\label{prop:angle-flat-uni}
For a flat unitary matrix $u \in \mn{n}$ with $n \geq 2$, we have
    $$\cos\alpha\left(\Delta^{(n)},u\Delta^{(n)}u^*\right) =\max_{1 \leq k \leq n} \frac{n - \Vert \Vec{u}_k \Vert_0}{(n-1)\Vert \Vec{u}_k \Vert_0}.$$
    Consequently, the value of this angle must belong to the finite discrete set:
    $$\left\{ \cos^{-1}\left(\frac{n - d}{(n-1)d} \right)\ \Bigg\vert{}\ d \in \{1, 2, \dots, n\} \right\} \subset \left[0, \frac{\pi}{2}\right]$$
\end{prop}

\begin{proof}
    Let $M = \vert u \vert^{\circ 2} \circ u - \frac{1}{n}u$. Then, for all $j,k=1,2,\dots,n$,
    $$M_{jk} = \vert u_{jk}\vert^2 u_{jk} - \frac{1}{n} u_{jk}
    = u_{jk} \left( \vert u_{jk}\vert^2 - \frac{1}{n} \right).$$

    Since $u$ is flat, we have
    $$M_{jk} = u_{jk} \left( \frac{1}{\Vert \Vec{u}_k \Vert_0} - \frac{1}{n} \right).$$
    Thus, $M=u\Lambda$, where $\Lambda$ is the diagonal matrix
    $$\Lambda = \operatorname{diag}\left( \frac{1}{\Vert \Vec{u}_1 \Vert_0} - \frac{1}{n}, \dots, \frac{1}{\Vert \Vec{u}_n \Vert_0} - \frac{1}{n} \right).$$
    Since $u$ is unitary, we have
    $\|M\|_{\operatorname{op}} = \|u\Lambda\|_{\operatorname{op}} = \|\Lambda\|_{\operatorname{op}}$.
    Therefore,
    \[
        \frac{n}{n-1}\|M\|_{\operatorname{op}}
        =
        \max_{1\leq k\leq n}
        \frac{n-\|\Vec{u}_k\|_0}
        {(n-1)\|\Vec{u}_k\|_0},
    \]
    and the result follows from the formula for
    $\cos\alpha\left(\Delta^{(n)},u\Delta^{(n)}u^*\right)$.

    To prove the last part let $d_{\min} = \min_{1 \le k \le n} \Vert \Vec{u}_k \Vert_0$. Because the function $f(x) = \frac{n-x}{(n-1)x}$ is strictly decreasing for $x > 0$ we get:
    $$\max_{1 \le k \le n} \frac{n - \Vert \Vec{u}_k \Vert_0}{(n-1)\Vert \Vec{u}_k \Vert_0} = \frac{n - d_{\min}}{(n-1)d_{\min}}$$
    Since $d_{\min} \in \{1, 2, \dots, n\}$, the quantity evaluates exactly to an element of the stated discrete set.
\end{proof}

We now show that the maximum angle between two masas is attained when the \emph{relative unitary} is a Hadamard unitary. We note that the equivalence $(ii)\iff(iii)$ in the following result was previously known; see \cite[Remark 5.4]{Bakshi_gupta_2021}.
\begin{theorem}\label{thm:angle-hadamard-comm-sq}
  Let $u,v\in\mn{n}$ be two $n\times n$ unitary matrices. Then the following are equivalent:
\begin{itemize}
    \item[(i)] $u^*v$ is a Hadamard unitary.
    
    \item[(ii)] $\alpha\left(u\Delta^{(n)}u^*,v\Delta^{(n)}v^*\right)=\frac{\pi}{2}$.
    
    \item[(iii)] The following diagram is a commuting square:
    \[
    \begin{array}{ccc}
    u\Delta^{(n)}u^* & \subset & \mn{n} \\
    \cup && \cup \\
    \mathbb{C} & \subset & v\Delta^{(n)}v^*
    \end{array}
    \]
\end{itemize}
\end{theorem}

\begin{proof}
    \noindent($(ii)\implies (i)$). For this let $w: = u^*v$. By Theorem \ref{thm:angle-betwn-masas} we get
\[
|w|^{\circ 2} \circ w = \frac{1}{n} w,  \text{ that is, } \left(|w_{ij}|^2 - \frac{1}{n}\right) w_{ij} = 0 \text{ for }i,j = 1, \dots, n.
\]
Now since $\sum_{j} |w_{ij}|^2=1$ for all $i = 1, \dots, n$ we conclude $|w_{ij}| = \frac{1}{\sqrt{n}}$ for all $i, j= 1\dots, n$.

    \noindent ($(i)\implies(ii)$). First observe that every $n\times n$ Hadamard unitary $w$ is a flat unitary satisfying $\|\Vec{w}_k\|_0=n$ for every column $k=1,2,\dots,n$. Hence,
$\|\overrightarrow{(u^*v)_k}\|_0=n$ for all $k=1,2,\dots,n$. The result now follows from Proposition~\ref{prop:angle-flat-uni} and Theorem~\ref{thm:uni-inv-angle}.

    \noindent ($(i) \iff (iii)$). This follows from Proposition \ref{prop:exp-and-hada-uni}.
\end{proof}

It is known that, in the $2\times 2$ case, $\cos \alpha(-,-)$ can attain every value in the interval $[0,1]$ (see \cite[Corollary 4.6]{gupta2024}). We now extend this result to the general $n\times n$ case:

\begin{theorem}\label{thm:range_angle}
    For every $r\in[0,1]$, there exist unitaries $u,v\in\mn{n}$ such that
\[
\cos\alpha\left(u\Delta^{(n)}u^*,v\Delta^{(n)}v^*\right)=r.
\]
\end{theorem}
\begin{proof}
Indeed, we shall prove that there always exists a unitary matrix $v\in \mn{n}$ such that
\[
    \cos\alpha(\Delta^{(n)}, v\Delta^{(n)}v^*)=r.
\]

To this end, define the function
\begin{align*}
    f : \sU(\mn{n}) &\longrightarrow [0, \infty)\\
    v &\longmapsto \frac{n^2}{n-1}
    \left\|
        \frac{1}{n}\left(|v|^{\circ 2} \circ v\right)
        - \frac{1}{n^2}v
    \right\|_{\mathrm{op}}.
\end{align*}

Note that, by  $\frac{n^2}{n-1}
    \left\|
        \frac{1}{n}\left(|v|^{\circ 2} \circ v\right)
        - \frac{1}{n^2}v
    \right\|_{\mathrm{op}}=  \cos\alpha(\Delta^{(n)}, v\Delta^{(n)}v^*)$

Fix a Hadamard unitary $v_0\in \mn{n}$; for example, let
\[
    (v_0)_{jk}=\frac{1}{\sqrt{n}}e^{\frac{2\pi ijk}{n}}
\]
for all $j,k$. By Theorem \ref{thm:angle-hadamard-comm-sq}, we have $f(v_0)=0$. On the other hand, $f(\mathbb{I}_n)=1$.

Since $\sU(\mn{n})$ is path connected and $f$ is clearly continuous, the intermediate value theorem implies that, for every $r\in[0,1]$, there exists a unitary matrix $v\in \sU(\mn{n})$ such that $f(v)=r$. This completes the proof.
\end{proof}

As an application of the formula for the interior angle established in Theorem \ref{thm:angle-betwn-masas}, we derive an explicit formula for the angle between two masas in $\mn2$, using the unitary invariance of $\alpha$ established in Theorem \ref{thm:uni-inv-angle}.
\begin{theorem}
   Let $u,v\in\sU(\mn{2})$ be arbitrary $2\times2$ unitary matrices. Then
\[
\cos\alpha\left(u\Delta u^*,v\Delta v^*\right)
=
\sqrt{1-4\left|(u^*v)_{11}\right|^2\left|(u^*v)_{12}\right|^2}.
\]
\end{theorem}
\begin{proof}
  By Theorem \ref{thm:uni-inv-angle}, we have
\[
\cos\alpha(u\Delta u^*,v\Delta v^*)
=
\cos\alpha(\Delta,u^*v\Delta v^*u).
\]
Applying Proposition \ref{prop:gupta-modified} to the right-hand side yields the desired result.
\end{proof}

We conclude this section by observing that the relationship between the Kadison--Kastler distance and the interior angle established in Theorem \ref{thm:dkk-is-sin} does not extend to $\mn{n}$ for $n \ge 3$.

\begin{remark}\label{rem:dkk-is-not-sin-dim-n}
    The equation $\dkk(\Deltan, u\Deltan  u^*) = \sin \alpha (\Deltan , u\Deltan  u^*)$ is not valid for $n \ge 3$. For this take the matrix $u = R_{jk}(\pi/4)$ and notice that this is not Hadamard unitary in $\mn{n}$. For this $u$, we have, from Theorem \ref{thm:angle-hadamard-comm-sq} that $\sin\alpha (\Deltan, u\Deltan u^*)\ne \sin\pi / 2=1$ whereas as we've observed in Remark \ref{rem:dkk-max-by-non-hada} that $\dkk (\Deltan, u\Deltan u^*) = 1.$ 
\end{remark}

    \section{Acknowledgment}
    Sumit Kumar gratefully acknowledges Prof. Keshab Chandra Bakshi for providing him with the opportunity to work at IIT Kanpur under his guidance through Project No. SPO/ANRF/MATH/2025271. He also sincerely thanks his PhD advisor Prof. Ved Prakash Gupta for introducing him to this area of research.

	\medskip
	
	\bibliographystyle{amsalpha} 
	\bibliography{references}

@article {gupta2024,
    AUTHOR = {Gupta, Ved Prakash and Sharma, Deepika},
     TITLE = {On possible values of the interior angle between intermediate
              subalgebras},
   JOURNAL = {J. Aust. Math. Soc.},
  FJOURNAL = {Journal of the Australian Mathematical Society},
    VOLUME = {117},
      YEAR = {2024},
    NUMBER = {1},
     PAGES = {44--66},
      ISSN = {1446-7887,1446-8107},
   MRCLASS = {46L05 (47L40)},
  MRNUMBER = {4802770},
       DOI = {10.1017/S1446788723000058},
       URL = {https://doi.org/10.1017/S1446788723000058},
}

@book {JonesSundar97,
    AUTHOR = {Jones, V. and Sunder, V. S.},
     TITLE = {Introduction to subfactors},
    SERIES = {London Mathematical Society Lecture Note Series},
    VOLUME = {234},
 PUBLISHER = {Cambridge University Press, Cambridge},
      YEAR = {1997},
     PAGES = {xii+162},
      ISBN = {0-521-58420-5},
   MRCLASS = {46L37 (46-01 46L10 57M25)},
  MRNUMBER = {1473221},
MRREVIEWER = {Carl\ Winsl\o w},
       DOI = {10.1017/CBO9780511566219},
       URL = {https://doi.org/10.1017/CBO9780511566219},
}

@article {BakshiGuin2025,
    AUTHOR = {Bakshi, Keshab Chandra and Guin, Satyajit},
     TITLE = {Relative position between a pair of spin model subfactors},
   JOURNAL = {J. Aust. Math. Soc.},
  FJOURNAL = {Journal of the Australian Mathematical Society},
    VOLUME = {119},
      YEAR = {2025},
    NUMBER = {1},
     PAGES = {1--38},
      ISSN = {1446-7887,1446-8107},
   MRCLASS = {46L37 (37A35 46L10 46L55)},
  MRNUMBER = {4929087},
MRREVIEWER = {Mainak\ Ghosh},
       DOI = {10.1017/S1446788725000035},
       URL = {https://doi.org/10.1017/S1446788725000035},
}

@article {Choda2008,
    AUTHOR = {Choda, Marie},
     TITLE = {Relative entropy for maximal abelian subalgebras of matrices
              and the entropy of unistochastic matrices},
   JOURNAL = {Internat. J. Math.},
  FJOURNAL = {International Journal of Mathematics},
    VOLUME = {19},
      YEAR = {2008},
    NUMBER = {7},
     PAGES = {767--776},
      ISSN = {0129-167X,1793-6519},
   MRCLASS = {46L55 (46L35)},
  MRNUMBER = {2437069},
MRREVIEWER = {E.\ St\o rmer},
       DOI = {10.1142/S0129167X08004881},
       URL = {https://doi.org/10.1142/S0129167X08004881},
}

@article {Choda2011,
    AUTHOR = {Choda, Marie},
     TITLE = {Conjugate pairs of subfactors and entropy for automorphisms},
   JOURNAL = {Internat. J. Math.},
  FJOURNAL = {International Journal of Mathematics},
    VOLUME = {22},
      YEAR = {2011},
    NUMBER = {4},
     PAGES = {577--592},
      ISSN = {0129-167X,1793-6519},
   MRCLASS = {46L55 (46L37)},
  MRNUMBER = {2794462},
MRREVIEWER = {Junsheng\ Fang},
       DOI = {10.1142/S0129167X1100691X},
       URL = {https://doi.org/10.1142/S0129167X1100691X},
}

@article {watatani1990,
    AUTHOR = {Watatani, Yasuo},
     TITLE = {Index for {$C^*$}-subalgebras},
   JOURNAL = {Mem. Amer. Math. Soc.},
  FJOURNAL = {Memoirs of the American Mathematical Society},
    VOLUME = {83},
      YEAR = {1990},
    NUMBER = {424},
     PAGES = {vi+117},
      ISSN = {0065-9266,1947-6221},
   MRCLASS = {46L05 (19K14 46L40 46L80 46M20)},
  MRNUMBER = {996807},
MRREVIEWER = {Masatoshi\ Enomoto},
       DOI = {10.1090/memo/0424},
       URL = {https://doi.org/10.1090/memo/0424},
}

@article {Kadison_Kastler,
    AUTHOR = {Kadison, Richard V. and Kastler, Daniel},
     TITLE = {Perturbations of von {N}eumann algebras. {I}. {S}tability of
              type},
   JOURNAL = {Amer. J. Math.},
  FJOURNAL = {American Journal of Mathematics},
    VOLUME = {94},
      YEAR = {1972},
     PAGES = {38--54},
      ISSN = {0002-9327,1080-6377},
   MRCLASS = {46L10},
  MRNUMBER = {296713},
MRREVIEWER = {H.\ Halpern},
       DOI = {10.2307/2373592},
       URL = {https://doi.org/10.2307/2373592},
}

@article {Ch_1974,
    AUTHOR = {Christensen, Erik},
     TITLE = {Perturbations of type {I} von {N}eumann algebras},
   JOURNAL = {J. London Math. Soc. (2)},
  FJOURNAL = {Journal of the London Mathematical Society. Second Series},
    VOLUME = {9},
      YEAR = {1974/75},
     PAGES = {395--405},
      ISSN = {0024-6107,1469-7750},
   MRCLASS = {46L10},
  MRNUMBER = {358373},
MRREVIEWER = {Arthur\ Lieberman},
       DOI = {10.1112/jlms/s2-9.3.395},
       URL = {https://doi.org/10.1112/jlms/s2-9.3.395},
}

@article {Ch_1977,
    AUTHOR = {Christensen, Erik},
     TITLE = {Perturbations of operator algebras},
   JOURNAL = {Invent. Math.},
  FJOURNAL = {Inventiones Mathematicae},
    VOLUME = {43},
      YEAR = {1977},
    NUMBER = {1},
     PAGES = {1--13},
      ISSN = {0020-9910,1432-1297},
   MRCLASS = {46L10},
  MRNUMBER = {512367},
MRREVIEWER = {Ole\ A.\ Nielsen},
       DOI = {10.1007/BF01390201},
       URL = {https://doi.org/10.1007/BF01390201},
}

@article {Ch_1980,
    AUTHOR = {Christensen, Erik},
     TITLE = {Near inclusions of {$C\sp{\ast} $}-algebras},
   JOURNAL = {Acta Math.},
  FJOURNAL = {Acta Mathematica},
    VOLUME = {144},
      YEAR = {1980},
    NUMBER = {3-4},
     PAGES = {249--265},
      ISSN = {0001-5962,1871-2509},
   MRCLASS = {46L05 (46L10)},
  MRNUMBER = {573453},
MRREVIEWER = {Man-Duen\ Choi},
       DOI = {10.1007/BF02392125},
       URL = {https://doi.org/10.1007/BF02392125},
}

@article {Ch_et_al_2010,
    AUTHOR = {Christensen, Erik and Sinclair, Allan and Smith, Roger R. and
              White, Stuart},
     TITLE = {Perturbations of {$C^\ast$}-algebraic invariants},
   JOURNAL = {Geom. Funct. Anal.},
  FJOURNAL = {Geometric and Functional Analysis},
    VOLUME = {20},
      YEAR = {2010},
    NUMBER = {2},
     PAGES = {368--397},
      ISSN = {1016-443X,1420-8970},
   MRCLASS = {46L05},
  MRNUMBER = {2671282},
       DOI = {10.1007/s00039-010-0070-y},
       URL = {https://doi.org/10.1007/s00039-010-0070-y},
}

@article {Bakshi_gupta_2021,
    AUTHOR = {Bakshi, Keshab Chandra and Gupta, Ved Prakash},
     TITLE = {Lattice of intermediate subalgebras},
   JOURNAL = {J. Lond. Math. Soc. (2)},
  FJOURNAL = {Journal of the London Mathematical Society. Second Series},
    VOLUME = {104},
      YEAR = {2021},
    NUMBER = {5},
     PAGES = {2082--2127},
      ISSN = {0024-6107,1469-7750},
   MRCLASS = {46L05 (46L37 47L40)},
  MRNUMBER = {4368671},
       DOI = {10.1112/jlms.12492},
       URL = {https://doi.org/10.1112/jlms.12492},
}

@article {Pimsner_Popa_1986,
    AUTHOR = {Pimsner, Mihai and Popa, Sorin},
     TITLE = {Entropy and index for subfactors},
   JOURNAL = {Ann. Sci. \'Ecole Norm. Sup. (4)},
  FJOURNAL = {Annales Scientifiques de l'\'Ecole Normale Sup\'erieure.
              Quatri\`eme S\'erie},
    VOLUME = {19},
      YEAR = {1986},
    NUMBER = {1},
     PAGES = {57--106},
      ISSN = {0012-9593},
   MRCLASS = {46L35 (46L55)},
  MRNUMBER = {860811},
MRREVIEWER = {Vaughan\ Jones},
       URL = {http://www.numdam.org/item?id=ASENS_1986_4_19_1_57_0},
}

@inproceedings{viro-topo,
  title={Elementary Topology: Problem Textbook},
  author={Oleg Ya. Viro and O. Ivanov and N. Yu. Netsvetaev and Viatcheslav Kharlamov},
  year={2008},
  url={https://api.semanticscholar.org/CorpusID:118275923}
}

@article {Gupta_kumar_2024,
    AUTHOR = {Gupta, Ved Prakash and Kumar, Sumit},
     TITLE = {On various notions of distance between subalgebras of operator algebras},
   JOURNAL = {M\"unster J. Math.},
  FJOURNAL = {M\"unster Journal of Mathematics},
    VOLUME = {17},
      YEAR = {2024},
    NUMBER = {1},
     PAGES = {241--272},
      ISSN = {1867-5778,1867-5786},
   MRCLASS = {46L05 (46L55 47L40)},
  MRNUMBER = {4976168},
       DOI = {10.17879/51958617466},
       URL = {https://doi.org/10.17879/51958617466},
}

@article {Watatani_ino_2014,
    AUTHOR = {Ino, Shoji and Watatani, Yasuo},
     TITLE = {Perturbations of intermediate {$\rm C^*$}-subalgebras for
              simple {$\rm C^*$}-algebras},
   JOURNAL = {Bull. Lond. Math. Soc.},
  FJOURNAL = {Bulletin of the London Mathematical Society},
    VOLUME = {46},
      YEAR = {2014},
    NUMBER = {3},
     PAGES = {469--480},
      ISSN = {0024-6093,1469-2120},
   MRCLASS = {46L05 (46L07 46L37)},
  MRNUMBER = {3210702},
MRREVIEWER = {Robert\ S.\ Doran},
       DOI = {10.1112/blms/bdu001},
       URL = {https://doi.org/10.1112/blms/bdu001},
}

@article{Kumar_2026,
author = {Kumar, Sumit},
title = {On stability of distance under some tensor products and some calculations},
journal = {Infinite Dimensional Analysis, Quantum Probability and Related Topics},
volume = {29},
number = {03},
pages = {2550016-1 to 2550016-17},
year = {2026},
doi = {10.1142/S021902572550016X},
URL ={https://doi.org/10.1142/S021902572550016X}
}

@article {bakshi-etal-2019,
    AUTHOR = {Bakshi, Keshab Chandra and Das, Sayan and Liu, Zhengwei and
              Ren, Yunxiang},
     TITLE = {An angle between intermediate subfactors and its rigidity},
   JOURNAL = {Trans. Amer. Math. Soc.},
  FJOURNAL = {Transactions of the American Mathematical Society},
    VOLUME = {371},
      YEAR = {2019},
    NUMBER = {8},
     PAGES = {5973--5991},
      ISSN = {0002-9947,1088-6850},
   MRCLASS = {46L37},
  MRNUMBER = {3937315},
MRREVIEWER = {Liguang\ Wang},
       DOI = {10.1090/tran/7738},
       URL = {https://doi.org/10.1090/tran/7738},
}

\end{document}